\documentclass[12pt]{amsart}
\usepackage[osf,sc]{mathpazo}
\usepackage{amssymb}

\usepackage{geometry}
\usepackage{bbm}
\usepackage{graphicx}
\usepackage{hyperref}
\hypersetup{
    colorlinks=true, 
    linktoc=all,     
    linkcolor=blue,
        citecolor=red,
    filecolor=black,
    urlcolor=blue	
}
\usepackage{enumerate}
\usepackage[inline]{enumitem}
\makeatletter
\newcommand{\inlineitem}[1][]{%
\ifnum\enit@type=\tw@
    {\descriptionlabel{#1}}
  \hspace{\labelsep}%
\else
  \ifnum\enit@type=\z@
       \refstepcounter{\@listctr}\fi
    \quad\@itemlabel\hspace{\labelsep}%
\fi} \makeatother
\newcommand{\ga}{\alpha}

\newcommand{\gn}{\nu}

\newcommand{\gp}{\pi}

\newcommand{\gs}{\sigma}

\newcommand{\gf}{\phi}

\newcommand{\gch}{\chi}

\newcommand{\Gd}{\Delta}

\newcommand{\subs}{\subset}

\newcommand{\bs}{\backslash}

\newcommand{\nin}{\notin}

\newcommand{\ti}{\tilde}

\newcommand{\mbb}{\mathbb}

\newcommand{\mcl}{\mathcal}

\newcommand{\ol}{\overline}

\newcommand{\us}{\underset}
\newcommand{\os}{\overset}

\newcommand{\lra}{\longrightarrow}

\newcommand{\N}{\mbb N}

\newcommand{\Ra}{\Rightarrow}

\newcommand{\es}{\emptyset}

\newcommand{\equ}[1]{%
\begin{equation*}
#1
\end{equation*}
}
\newcommand{\equa}[1]{%
\begin{equation*}
\begin{aligned}
#1
\end{aligned}
\end{equation*}
}
\newcommand{\equan}[2]{%
\begin{equation}
\label{Eq:#1}
\begin{aligned}
#2
\end{aligned}
\end{equation}
}
\DeclareMathOperator{\Det}{Det}

\DeclareMathOperator{\Disc}{Disc}

\newcommand{\mattwo}[4]{%
\begin{pmatrix}
  #1 & #2\\ #3 & #4
\end{pmatrix}
}

\newcommand{\matcolthree}[3]{%
\begin{pmatrix}
  #1\\#2\\#3
\end{pmatrix}
}

\newcommand{\matthree}[9]{%
\begin{pmatrix}
  #1 & #2 & #3\\ #4 & #5 & #6\\ #7 & #8 & #9
\end{pmatrix}
}

\newtheorem{theorem}{Theorem}[section]

\newtheorem{prop}[theorem]{Proposition}
\newtheorem{claim}[theorem]{Claim}
\newtheorem{lemma}[theorem]{Lemma}
\newtheorem{cor}[theorem]{Corollary}

\theoremstyle{definition}
\newtheorem{defn}[theorem]{Definition}
\newtheorem{example}[theorem]{Example}

\theoremstyle{remark}
\newtheorem{note}[theorem]{Note}
\newtheorem{remark}[theorem]{Remark}
\numberwithin{equation}{section}
\makeatletter
\def\namedlabel#1#2{\begingroup
   \def\@currentlabel{#2}%
   \label{#1}\endgroup
}
\makeatother
\newtheorem*{thmA}{\bf{Theorem A}}
\newtheorem*{thmB}{\bf{Theorem B}}
\begin{document}
\title[On Very Generic Discriminantal Arrangements]{On Very Generic Discriminantal Arrangements}
\author[C.P. Anil Kumar]{Author: C.P. Anil Kumar*}
\address{Post Doctoral Fellow in Mathematics, Harish-Chandra Research Institute, Chhatnag Road, Jhunsi, Prayagraj (Allahabad)-211019, Uttar Pradesh, INDIA
}
\email{akcp1728@gmail.com}
\thanks{*The work is done when the author is a Post Doctoral Fellow at HRI, Allahabad.}
\date{\sc \today}
\begin{abstract}
In this article we prove two main results. Firstly,	we show that any six-line arrangement, consisting of three pairs of mutually perpendicular lines, does not give rise to a ``very generic or sufficiently general" discriminantal arrangement in the sense of C.~A.~Athanasiadis~\cite{MR1720104}. We give two proofs of the first result. The second result is as follows. The codimension-one boundary faces of (a region) a convex cone of a very generic discriminantal arrangement has not been characterised and is not known even though the intersection lattice of a very generic discriminantal arrangement is known. So secondly, we show that the number of simplex cells of the very generic hyperplane arrangement $\mcl{H}^m_n=\{H_i:\us{j=1}{\os{m}{\sum}}a_{ij}x_j=c_i,1\leq i\leq n\}$ may not be not precisely equal to the number of codimension-one boundary hyperplanes of 
$\mbb{R}^n$ of the convex cone $C$ containing $(c_1,c_2,\ldots,c_n)$ in the associated very generic discriminantal arrangement. That is, for $1\leq i_1<i_2<\ldots<i_m<i_{m+1}\leq n$, if $\Gd^m H_{i_1}H_{i_2}\ldots H_{i_m}H_{i_{m+1}}$ is a simplex cell of the hyperplane arrangement $\mcl{H}^m_n$ then it need not give rise to a codimension-one boundary hyperplane of the convex cone $C$ containing $(c_1,c_2,\ldots,c_n)$ in the associated very generic discriminantal arrangement. We finally mention an interesting open-ended remark before the appendix section.

In the appendix section we give a self contained exposition and describe combinatorially the intersection lattice of a (Zariski open and dense) class of ``very generic or sufficiently general" discriminantal arrangements. An important ingredient involved in this exposition is the Crapo's characterization (~\cite{MR0843374}, Section 6, Page 149, Theorem 2) of the matroid of circuits of the configuration of $n$ ``generic" points in $\mbb{R}^m$ as an application of Hall's marriage theorem and a result in linear algebra about a minor and its complementary minor of an orthogonal matrix. We mention this characterization here in detail. Another aspect of this exposition is Conjecture $4.3$ stated in M.~Bayer and K.~Brandt~\cite{MR1456579} and its proof which is given in C.~A.~Athanasiadis~\cite{MR1720104}. We mention this aspect also in detail. As a consequence , we give a geometric description of the lattice elements as sets of concurrencies of the hyperplane arrangements which give the same  ``very generic or sufficiently general" discriminantal arrangement. 
\end{abstract}
\maketitle
\section{\bf{Introduction}}
{\it Concurrency Geometries} have been studied by H.~H.~Crapo in~\cite{MR0766269} in order to solve problems on configurations of hyperplane arrangements in a Euclidean space. Imagine a finite set $\mcl{H}_n^m$ of $n$-hyperplanes in $\mbb{R}^m$ for some $n,m\in \N$ which can move freely and whose normal directions arise from a fixed finite set $\mcl{N}\subs \mbb{R}^m$ of cardinality $n$ which is generic, that is, any subset $\mcl{M}\subseteq \mcl{N}$ of cardinality at most $m$ is linearly independent. They give rise to different hyperplane arrangements in $\mbb{R}^m$. The combinatorial aspects of such configurations of hyperplane arrangements is studied by associating with them another arrangement which is known in the literature as {\it Discriminantal arrangements} or {\it Manin-Schechtman arrangements}  (refer to Page 205, Section 5.6 in P.~Orlik and H.~Terao~\cite{MR1217488}). Some of the authors who have worked on the discriminantal arrangements are  C.~A.~Athanasiadis~\cite{MR1720104}, M.~Bayer and K.~Brandt~\cite{MR1456579}, M.~Falk~\cite{MR1209098}, Yu.~I.~Manin and V.~V.~Schechtman~\cite{MR1097620} and more recently A.~Libgober and S.~Settepanella~\cite{MR3899551}. 

In this paper we prove in first main Theorem~\ref{theorem:SixLines} that, a six-line arrangement, consisting of three pairs of mutually perpendicular lines, does not give rise to a ``sufficiently general" discriminantal arrangement in the sense of C.~A.~Athanasiadis~\cite{MR1720104}.

In second main Theorem~\ref{theorem:SixLinesII}, we prove that there exists a generic six-line arrangement $\mcl{H}^2_6=\{H_i: a_{i1}x_1+a_{i2}x_2=c_i,1\leq i\leq 6\}$ which give rise to very generic discriminantal arrangement $\mcl{C}^6_{\binom{6}{3}}=\{M_{\{i_1,i_2,i_3\}}\mid 1\leq i_1<i_2<i_3\leq 6\}$ such that not every triangular cell (2-simplex region) of $\mcl{H}^2_6$ gives rise to a codimension-one boundary of the convex cone $C$ containing $(c_1,c_2,c_3,c_4,c_5,c_6)$ in the associated very generic discriminantal arrangement $\mcl{C}^6_{\binom{6}{3}}$. 
\section{\bf{Definitions and Statements of the main results}}
We begin with some definitions before stating the main results.
\begin{defn}[A Hyperplane Arrangement, An Essential Hyperplane Arrangement, A Generic Hyperplane Arrangement, A Central Hyperplane Arrangement]
	\label{defn:HA}
	~\\
	Let $m,n$ be positive integers. We say a set 
	\equ{\mcl{H}_n^m=\{H_1,H_2,\ldots,H_n\}} of $n$ affine hyperplanes in $\mbb{R}^m$ forms a 
	hyperplane arrangement. We say that the hyperplane arrangement $\mcl{H}_n^m$ is essential if the normals of the hyperplanes $H_i,\ 1\leq i\leq n$ span $\mbb{R}^m$. We say that they form a generic hyperplane arrangement or a hyperplane arrangement in general position, if Conditions 1,2 are satisfied.
	\begin{itemize}
		\item Condition 1: For $1\leq r \leq m$, the intersection of any $r$ hyperplanes has dimension $m-r$.
		\item Condition 2: For $r>m$, the intersection of any $r$ hyperplanes is empty.
	\end{itemize}
	We say that the hyperplane arrangement $\mcl{H}_n^m$ is central if $\us{i=1}{\os{n}{\cap}}H_i\neq \es$.
\end{defn}
\begin{remark}
A very generic hyperplane arrangement is defined below in Definition~\ref{defn:GenericVeryGeneric} and more specifically in Definition~\ref{defn:GenericVeryGenericArrangements}.
\end{remark}
\begin{defn}[Discriminantal Arrangement-A Central Arrangement]
	\label{defn:CA}
	~\\
	Let $m,n$ be positive integers. Let \equ{\mcl{H}_n^m=\{H_1,H_2,\ldots,H_n\}} 
	be a generic hyperplane arrangement of $n$ hyperplanes in $\mbb{R}^m$. Let the equation for $H_i$ be given by 
	\equ{\us{j=1}{\os{m}{\sum}}a_{ij}x_j=b_i,\text{ with } a_{ij}, b_i\in \mbb{R}, 1\leq j\leq m, 1\leq i\leq n.}
	For every $1\leq i_1<i_2<\ldots<i_{m+1}\leq n$ consider the hyperplane $M_{\{i_1,i_2,\ldots,i_{m+1}\}}$ 
	passing through the origin in $\mbb{R}^n$ in the variables $y_1,y_2,\ldots,y_n$ whose equation is given 
	by 
	\equ{\Det
		\begin{pmatrix}
			a_{i_11} & a_{i_12} & \cdots & a_{i_1(m-1)} & a_{i_1m} & y_{i_1}\\
			a_{i_21} & a_{i_22} & \cdots & a_{i_2(m-1)} & a_{i_2m} & y_{i_2}\\
			\vdots   & \vdots   & \ddots & \vdots       & \vdots   & \vdots\\
			a_{i_{m-1}1} & a_{i_{m-1}2} & \cdots & a_{i_{m-1}(m-1)} & a_{i_{m-1}m} & y_{i_{m-1}}\\
			a_{i_m1} & a_{i_m2} & \cdots & a_{i_m(m-1)} & a_{i_mm} & y_{i_m}\\
			a_{i_{m+1}1} & a_{i_{m+1}2} & \cdots & a_{i_{m+1}(m-1)} & a_{i_{m+1}m} & y_{i_{m+1}}\\
		\end{pmatrix}
		=0}
	Then the associated discriminantal arrangement of hyperplanes passing through the origin in $\mbb{R}^n$ 
	is given by
	\equ{\mcl{C}^n_{\binom{n}{m+1}}=\{M_{\{i_1,i_2,\ldots,i_{m+1}\}}\mid 1\leq 
		i_1<i_2<\ldots<i_{m+1}\leq n\}.}
	It is a central arrangement consisting of hyperspaces, that is, linear subspaces of codimension one in $\mbb{R}^n$. 
\end{defn}
\begin{note}
	Even though the definition of hyperplanes $M_{\{i_1,i_2,\ldots,i_{m+1}\}}$ of the discriminantal arrangement $\mcl{C}^n_{\binom{n}{m+1}}$ involves the coefficients $[a_{ij}]_{1\leq j\leq m,1\leq i\leq n}$ of the variables $x_i,1\leq i\leq m$ 
	we can pick and fix any one set of equations for the hyperplanes $H_i, 1\leq i \leq n$ of the hyperplane arrangement to associate the discriminantal arrangement.
\end{note}
\begin{note}[Convention: Fixing the coefficient matrix]
	\label{note:Convention}
Let $m,n$ be positive integers. Let $\mcl{U}=\{v_i=(a_{i1},a_{i2},\ldots,a_{im})\mid 1\leq i\leq n\}$ be a finite set of vectors in $\mbb{R}^m$ such that any subset $\mcl{V}\subseteq \mcl{U}$ of cardinality at most $m$ is a linearly independent set. We fix the  matrix $[a_{ij}]_{1\leq i\leq n,1\leq j\leq m}\in M_{n\times m}(\mbb{R})$. 
	Let $\mcl{H}_n^m=\{H_1,H_2,\ldots,H_n\}$ be any hyperplane arrangement such that the normal vector of $H_i$ is $v_i,1\leq i\leq n$. When we write equations for the hyperplane $H_i$, we use the fixed matrix and write 
	\equ{H_i:\us{j=1}{\os{m}{\sum}} a_{ij}x_j=b_i \text{ for some }b_i\in \mbb{R}.}
	Using this coefficient matrix, we define the discriminantal arrangement $\mcl{C}^n_{\binom{n}{m+1}}$ which depends only on set $\mcl{U}$. Any such hyperplane arrangement gives a vector $(b_1,b_2,\ldots,b_n)$ and conversely any vector $(b_1,b_2,\ldots,b_n)$ gives such a hyperplane arrangement $\mcl{H}_n^m$.  If the vector $(b_1,b_2,\ldots,b_n)$ lies in the interior of a cone of the discriminantal arrangement $\mcl{C}^n_{\binom{n}{m+1}}$, then the hyperplane arrangement $\mcl{H}_n^m$ is generic. The combinatorics of such arrangements $\mcl{H}_n^m$ is studied by using the discriminantal arrangement $\mcl{C}^n_{\binom{n}{m+1}}$. 
\end{note}
The combinatorics of the discriminantal arrangement -- more precisely the intersection lattice of the discriminantal arrangement -- has been studied very well. The intersection lattice for a certain class of ``sufficiently general" (Definition $2.1$ in~\cite{MR1720104}) discriminantal arrangements which maximises the $f$-vector of the intersection lattice has been already characterised by C.~A.~Athanasiadis (refer to Theorem $2.3$ in~\cite{MR1720104}). An important ingredient in its proof is the Crapo's characterisation of the matroid $M(n,m)$ of circuits of the configuration of $n$-generic points in $\mbb{R}^m$. This matroid is introduced in H.~H.~Crapo~\cite{MR0766269} and characterised in H.~H.~Crapo~\cite{MR0843374}, Chapter 6, when the coordinates of the $n$-points are generic indeterminates, as the Dilworth completion of $D_m(B_n)$ of the $m^{th}$-lower truncation of the Boolean algebra of rank $n$ (see H.~H.~Crapo and G.~C.~Rota~\cite{MR0290980}, Chapter 7). The intersection lattice of ``sufficiently general" discriminantal arrangements coincides with the lattice $L(n,m)$ of flats of $M(n,m)$. In C.~A.~Athanasiadis~\cite{MR1720104}, it is proved that this lattice is isomorphic to the lattice $P(n,m)$ (refer to Theorem $3.2$ in~\cite{MR1720104}).  $P(n,m)$ is the collection of all sets of the form $\mcl{S}=\{S_1,S_2,\ldots,S_r\}$, where $S_i\subseteq\{1,2,\ldots,n\}$, each of cardinality at least $m+1$, such that 
\equ{\mid \us{i\in I}{\bigcup}S_i \mid>m+\us{i\in I}{\sum}(\mid S_i\mid-m)}
for all $I\subseteq \{1,2,\ldots,r\}$ with $\mid I\mid \geq 2$.  They partially order $P(n,m)$ by letting $\{S_1,S_2,\ldots,S_r\}=\mcl{S}\leq \mcl{T}=\{T_1,T_2,\ldots,T_p\}$, if, for each $1\leq i\leq r$ there exists $1\leq j\leq p$ such that $S_i\subseteq T_j$.  This isomorphism between the lattices $L(n,m)$ and $P(n,m)$ was initially conjectured (refer to Definition $4.2$ and Conjecture $4.3$) in M.~Bayer and K.~Brandt~\cite{MR1456579}.

\begin{defn}[Generic Sets, Very Generic Sets and Very Generic Hyperplane Arrangements]
\label{defn:GenericVeryGeneric}
Let $m,n$ be positive integers. A finite set $\mcl{U}=\{v_i=(a_{i1},a_{i2},\ldots,a_{im})\mid 1\leq i\leq n\}\subs \mbb{R}^m$ is said to be generic if any subset $\mcl{V}\subseteq \mcl{U}$ of cardinality at most $m$ is a linearly independent set. A generic set $\mcl{U}$ is said to be very generic if it gives rise to a ``sufficiently general" discriminantal arrangement in sense of C.~A.~Athanasiadis (Definition $2.1$ in~\cite{MR1720104}). Also refer to Definition~\ref{defn:GenericVeryGenericArrangements} for more details. This implies that the discriminantal arrangement is ``very generic" in the sense of M.~Bayer and K.~Brandt (Definition $4.2$  in~\cite{MR1456579}) by using Theorem 2.3 in~\cite{MR1720104}. A generic hyperplane arrangement $\mcl{H}_n^m=\{H_i:\us{j=1}{\os{m}{\sum}} a_{ij}x_j=b_i\mid 1\leq i\leq n\}$ in $\mbb{R}^m$ is very generic if  $\mcl{U}=\{v_i=(a_{i1},a_{i2},\ldots,a_{im})\mid 1\leq i\leq n\}\subs \mbb{R}^m$ is a very generic set.
The space $\mcl{O}(n,m)$ of very generic hyperplane arrangements is zariski open and dense in the space of all generic hyperplane arrangements. 
\end{defn}
We state the main theorems.
\begin{thmA}
\namedlabel{theorem:SixLines}{A}
Let $\mcl{M}=\{v_i\mid 1\leq i\leq 6\}\subs \mbb{R}^2$ be a generic set such that the set $\mcl{M}$ gives rise to six-line arrangements where there are three pairs of mutually perpendicular lines. Then $\mcl{M}$ is not very generic, that is, it gives rise to a discriminantal arrangement which is not a ``sufficiently general" discriminantal arrangement in the sense of C.~A.~Athanasiadis (Definition $2.1$ in~\cite{MR1720104}).
\end{thmA}
\begin{remark}
	\label{remark:FirstTheorem}
We will give two proofs of Theorem~\ref{theorem:SixLines}. The first one uses elementary techniques and only uses the fact that the space of ``sufficiently general" discriminantal arrangements is an open set in the space of discriminantal arrangements associated to line arrangements where the lines are in general position and a topological fact in Remark~\ref{remark:TopFact} which can be deduced. This proof also reveals much of the content of Theorem~\ref{theorem:SixLines}. The second proof uses the result of C.~A.~Athanasiadis~\cite{MR1720104} that the intersection lattice of a "sufficiently general" discriminantal arrangement is isomorphic to the lattice $P(n,m)$ (refer to Theorem $3.2$ in~\cite{MR1720104}). This proof does not reveal much, the content of Theorem~\ref{theorem:SixLines}.
\end{remark}
\begin{thmB}
	\namedlabel{theorem:SixLinesII}{B}
	There exists a generic six-line arrangement $\mcl{H}^2_6=\{H_i: a_{i1}x_1+a_{i2}x_2=c_i,1\leq i\leq 6\}$ which give rise to a very generic or a sufficiently general discriminantal arrangement $\mcl{C}^6_{\binom{6}{3}}=\{M_{\{i_1,i_2,i_3\}}\mid 1\leq i_1<i_2<i_3\leq 6\}$, that is, the set $\mcl{V}=\{(a_{i1},a_{i2})\mid 1\leq i\leq 6\}$ is a very generic set, such that not every triangular cell (2-simplex region) of $\mcl{H}^2_6$ gives rise to a codimension-one boundary of the convex cone $C$ containing $(c_1,c_2,c_3,c_4,c_5,c_6)$ in the associated very generic discriminantal arrangement $\mcl{C}^6_{\binom{6}{3}}$.  
\end{thmB}
\section{\bf{First Proof of the First Main Theorem}}
We begin with a lemma.
\begin{lemma}
Let $0<m_1<m_2$ be two real numbers. Let $\mcl{M}=\{v_1=(0,1),v_2=(-\frac 1{m_1},1),v_3=(-\frac 1{m_2},1),v_4=(1,0),v_5=(m_1,1),v_6=(m_2,1)\}$ be a generic set containing six vectors in $\mbb{R}^2$. Then for any generic six-line arrangement $\mcl{L}_6^2=\{L_1,L_2,L_3,L_4,L_5,L_6\}$ such that a normal vector of $L_i$ is $v_i$, there exists three subscripts $1\leq i<j<k\leq 6$ such that the triangle $\Gd L_iL_jL_k$, that is, the triangle formed by the lines or sides $L_i,L_j,L_k$ is an acute-angled triangle. 
\end{lemma}
\begin{proof}
We can assume that the line $L_1$ is the X-axis, $L_4$ is the Y-axis, $L_i$ has slope $m_{i-1}$ for $i=2,3$, $L_j$ has slope $-\frac 1{m_{j-4}}$ for $j=5,6$. So $L_i\perp L_{i+3}$ for $i=1,2,3$. Let $A=L_1\cap L_2,B=L_1\cap L_6, C=L_2\cap L_6, D= L_4\cap L_2, E=L_4\cap L_6$. Then the angles $\measuredangle CAB, \measuredangle ABC, \measuredangle CDE, \measuredangle CED$ are acute angles. If $\measuredangle ACB$ is also acute then the triangle $\Gd L_1L_2L_6$ is an acute-angled triangle. If $\measuredangle ACB$ is obtuse then the angle $\measuredangle DCE$ is acute. Hence the triangle $\Gd L_4L_2L_6$ is an acute-angled triangle. This proves the lemma.	
\end{proof}
Now we have the corollary given below follows using the above lemma.
\begin{cor}
	\label{cor:Acute}
In any generic six-line arrangement with three pairs of mutually perpendicular lines, the lines can be individually translated to form an arrangement of six lines such that some three lines form an acute-angled triangle and the remaining three lines are the altitudes of the triangle.  
\end{cor}
Now we prove two propositions regarding such six-line arrangements.
\begin{prop}
Consider Figure~\ref{fig:One} which is obtained by translating the altitudes of the $\Gd L_1L_3L_5$  to $L_4,L_6,L_2$ in a manner given in the figure. 
	\begin{figure}[h]
	\centering
	\includegraphics[width = 0.4\textwidth]{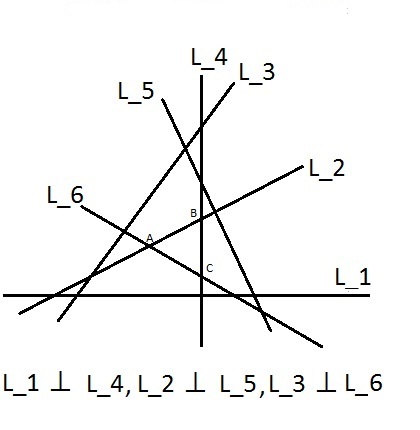}
	\caption{A Generic Six-Line Arrangement with Three Pairs of Mutually Perpendicular Lines}
	\label{fig:One}
\end{figure}
Then the distance of the perpendicular dropped from $A=L_2\cap L_6$ to the line $L_4$ is greater than the distance of the perpendicular dropped from $L_3\cap L_5$ to the line $L_4$. The distance of the perpendicular dropped from $B=L_2\cap L_4$ to the line $L_6$ is greater than the distance of the perpendicular dropped from $L_1\cap L_5$ to the line $L_6$. The distance of the perpendicular dropped from $C=L_4\cap L_6$ to the line $L_2$ is greater than the distance of the perpendicular dropped from $L_1\cap L_3$ to the line $L_2$.  	
\end{prop}
\begin{proof}
We prove that, the distance of the perpendicular dropped from $A=L_2\cap L_6$ to the line $L_4$ is greater than the distance of the perpendicular dropped from $L_3\cap L_5$ to the line $L_4$. The proof of the remaining two assertions are similar.

Let $L_1$ be the X-axis, $L_4$ be the Y-axis. Let the equation of $L_i$ be $y=m_{i-1}x+c_{i-1}$ for $i=2,3$ for some real numbers $m_1,m_2,c_1,c_2$. Then the equation of $L_j$ is $y=-\frac 1{m_{j-4}}x+d_{j-4}$ for $j=5,6$ and some real numbers $d_1,d_2$.
Now we observe from the figure that $0<m_1<m_2$ and using the points on the Y-axis we have $c_2>d_1>c_1>d_2>0$. Also we have $-\frac {c_1}{m_1}<-\frac {c_2}{m_2}<0<d_2m_2<d_1m_1$ using the points on the X-axis. The distance of the perpendicular dropped from $A=L_2\cap L_6$ to the line $L_4$ is $\frac{m_2(c_1-d_2)}{1+m_1m_2}$. The distance of the perpendicular dropped from $L_3\cap L_5$ to the line $L_4$ is $\frac{m_1(c_2-d_1)}{1+m_1m_2}$. Now we observe the following.
\equa{&\frac {c_1}{m_1}-\frac {c_2}{m_2}>0>\frac{d_2m_2-d_1m_1}{m_1m_2}=\frac{d_2}{m_1}-\frac{d_1}{m_2} \Ra \frac{c_1-d_2}{m_1}>\frac{c_2-d_1}{m_2}\Ra\\ &m_2(c_1-d_2)>m_1(c_2-d_1) \Ra \frac{m_2(c_1-d_2)}{1+m_1m_2}>\frac{m_1(c_2-d_1)}{1+m_1m_2}.}
This proves the proposition.
\end{proof}	
\begin{prop}
	\label{prop:OrderofPoints}
Consider Figure~\ref{fig:One} and Figure~\ref{fig:Two}.
\begin{figure}[h]
	\centering
	\includegraphics[width = 0.4\textwidth]{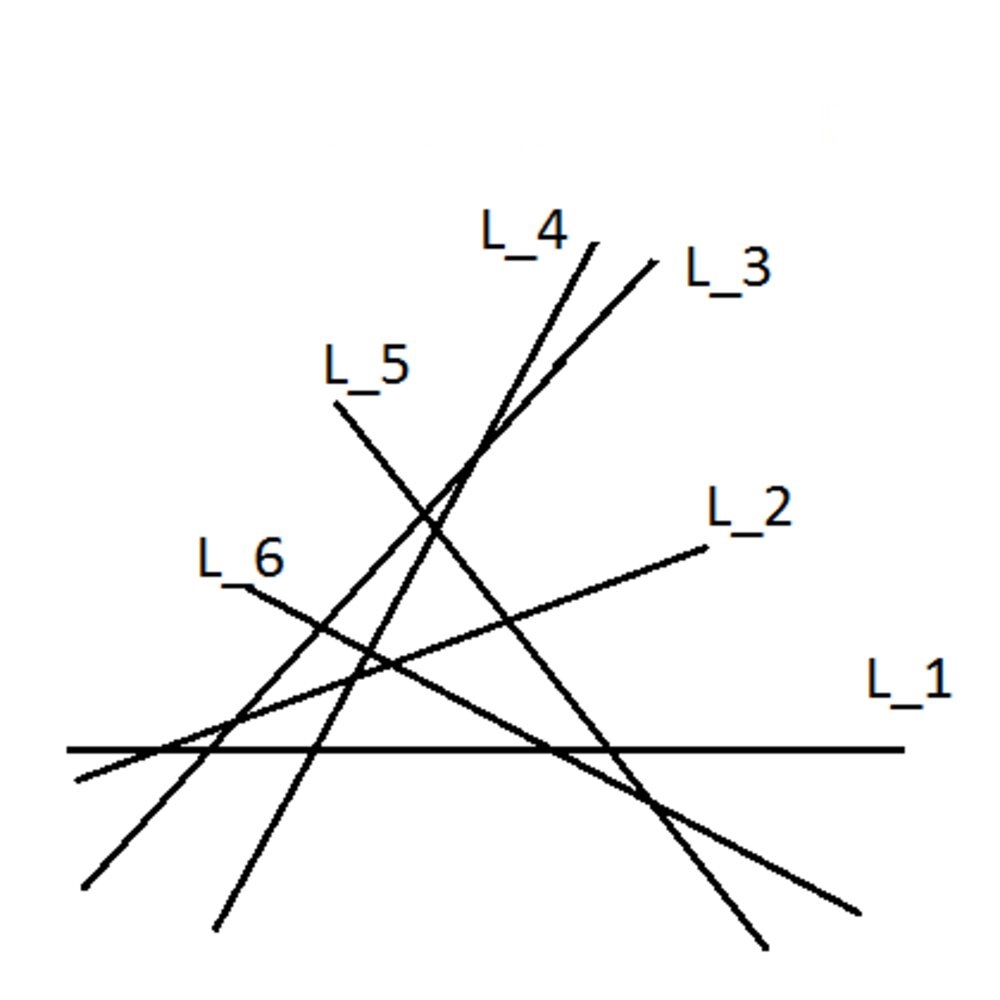}
	\caption{A Generic Six-Line Arrangement}
	\label{fig:Two}
\end{figure}
Suppose the order of points 
\begin{itemize}
	\item on $L_1$ is given by $L_2\cap L_1, L_3\cap L_1, L_4\cap L_1, L_6\cap L_1,L_5\cap L_1$,
	\item on $L_3$ is given by $L_1\cap L_3,L_2\cap L_3,L_6\cap L_3, L_5\cap L_3,L_6\cap L_3$,
	\item on $L_5$ is given by $L_3\cap L_5,L_4\cap L_5,L_2\cap L_5, L_1\cap L_5,L_6\cap L_5$.
\end{itemize}
If $L_1\perp L_4,L_2\perp L_5,L_3\perp L_6$ then, as given in Figure~\ref{fig:One}, the order of points 
\begin{itemize}
	\item on $L_2$ must be $L_1\cap L_2, L_3\cap L_2, L_6\cap L_2, L_4\cap L_2, L_5\cap L_2$
	\item on $L_4$ must be $L_1\cap L_4, L_6\cap L_4, L_2\cap L_4, L_5\cap L_4, L_3\cap L_4$
	\item on $L_6$ must be $L_3\cap L_6, L_2\cap L_6, L_4\cap L_6, L_1\cap L_6, L_5\cap L_6$.
\end{itemize}
The order of points on lines $L_2,L_4,L_6$ as given in Figure~\ref{fig:Two} is not possible.
\end{prop}
\begin{proof}
It is easy to see that given the order of points on $L_1,L_3,L_5$, the points of intersections $L_2\cap L_4,L_4\cap L_6, L_6 \cap L_2$ lie in the interior of the triangle $\Gd L_1L_3L_5$. This can be proved as follows. Consider the convex quadrilateral $\square PQRS$ whose vertices are $P=L_4\cap L_5, Q=L_2\cap L_5,R=L_1\cap L_4,S=L_3\cap L_2$. This quadrilateral lies inside the triangle $\Gd L_1L_3L_5$ and the lines $L_2,L_4$ form the diagonals of the quadrilateral. Hence $L_2\cap L_4$ lies in the interior of the triangle $\Gd L_1L_3L_5$. The cases of the points $L_4\cap L_6,L_2\cap L_6$ are similar.

Now there are two possibilities for orders of points on lines $L_2,L_4,L_6$ as shown in Figure~\ref{fig:One} and Figure~\ref{fig:Two}. We need to show that if $L_1\perp L_4,L_2\perp L_5,L_3\perp L_6$ then the orders of points on lines $L_2,L_4,L_6$ given in Figure~\ref{fig:One} occur and the orders of points on lines $L_2,L_4,L_6$ given in Figure~\ref{fig:Two} does not occur.

We prove this as follows. Assume that $L_1\perp L_4,L_2\perp L_5,L_3\perp L_6$ and $L_1$ is the X-axis and $L_4$ is the Y-axis. Assume without loss of generality by using a reflection about the X-axis, that the triangle $\Gd L_1L_3L_5$ is in the closed upper half-plane. 
Now assume without loss of generality by using a reflection about the Y-axis, that the extreme point $L_1\cap L_5$ on $L_1$ lies on the positive X-axis and the extreme point $L_1\cap L_2$ on $L_1$ lies on the negative X-axis. Note that under reflections, the orders of intersection points on the lines do not change.

Now the point $L_2\cap L_5$ lies in the Quadrant I, the point $L_5\cap L_6$ lies in Quadrant IV, the points $L_3\cap L_5,L_3\cap L_6,L_3\cap L_2$ lie in Quadrant II, the points $L_2\cap L_1,L_3\cap L_1$ lie on the negative X-axis, the points $L_1\cap L_6,L_1\cap L_5$ lie on the positive X-axis. The points $L_i\cap L_4$ for $i=2,3,5,6$ lie on the positive Y-axis.

Now we observe that the order of intersections on a large circle $C$ enclosing all the intersection points $L_i\cap L_j, 1\leq i<j\leq 6$ in an anticlockwise manner is the $L_1\cap C, L_2\cap C, L_3\cap C, L_4\cap C, L_5\cap C, L_6 \cap C$.  
Also if we let the equation of $L_i$ be given by $y=m_{i-1}x+c_{i-1}$ for $i=2,3$ then we have $0<m_1<m_2,0<c_1<c_2$ and the equation of the line $L_j$ is given by $y=-\frac 1{m_{j-4}}x+d_{j-4}$ for $j=5,6$ for some $0<d_2<d_1$.

Using the order of intersection points on the X-axis, we get $-\frac {c_1}{m_1}<-\frac {c_2}{m_2}<0<d_2m_2<d_1m_1$. Using the orders of intersection points on the lines $L_3$ and $L_5$, we also have $c_2>d_1>\max(c_1,d_2)>0$. Now we observe
\equa{&\frac {c_1}{m_1}-\frac {c_2}{m_2}>0>\frac{d_2m_2-d_1m_1}{m_1m_2}=\frac{d_2}{m_1}-\frac{d_1}{m_2} \Ra \frac{c_1-d_2}{m_1}>\frac{c_2-d_1}{m_2}\Ra\\ &m_2(c_1-d_2)>m_1(c_2-d_1) \Ra \frac{m_2(c_1-d_2)}{1+m_1m_2}>\frac{m_1(c_2-d_1)}{1+m_1m_2}>0.}
So we obtain that $c_1>d_2$ and the point $L_2\cap L_6$ lies in the Quadrant II since its X-coordinate is $-\frac{m_2(c_1-d_2)}{1+m_1m_2}$ which is negative. 

Now we conclude that the orders of intersection points on the lines $L_2,L_4,L_6$ is as given in Figure~\ref{fig:One} and not as given in Figure~\ref{fig:Two}. This proves the proposition. 
\end{proof}
Now we prove the first main theorem after a remark.
\begin{remark}
\label{remark:TopFact}
The main idea behind the first proof of Theorem~\ref{theorem:SixLines} is that if a line arrangement $\mcl{L}_n^2$ arises from a discriminatal arrangement $\mcl{C}^n_{\binom{n}{3}}$ which is ``sufficiently general" which in turn is an open condition, a line arrangement $\ti{\mcl{L}}_n^2$ obtained by small perturbations of the slopes of the lines in $\mcl{L}_n^2$ must also arise from a ``sufficiently general" discriminantal arrangement $\ti{\mcl{C}}^n_{\binom{n}{3}}$. Moreover, all the finitely many configurations of line arrangements arising from $\mcl{C}^n_{\binom{n}{3}}$ are also given by the line arrangements arising from $\ti{\mcl{C}}^n_{\binom{n}{3}}$ and vice versa.  
\end{remark}

\begin{proof}[Proof of Theorem~\ref{theorem:SixLines}]
Consider the open set $\mcl{U}(6,2)$ in $(\mbb{R}^3)^6$ given by \equa{\mcl{U}(6,2)=\{((a_i,b_i,c_i))_{i=1}^6\in (\mbb{R}^3)^6\mid &\Det\matthree{a_i}{b_i}{c_i}{a_j}{b_j}{c_j}{a_k}{b_k}{c_k}\neq 0,1\leq i<j<k\leq 6,\\&\Det\mattwo {a_i}{b_i}{a_j}{b_j}\neq 0, 1\leq i<j\leq 6\}.}
If $((a_i,b_i,c_i))_{i=1}^6\in \mcl{U}(6,2)$ then the line arrangement $\mcl{L}^2_6=\{L_i:a_ix+b_iy=c_i\mid 1\leq i\leq 6\}$ is a generic line arrangement. Let $\mcl{V}(6,2)\subs \mcl{U}(6,2)$ denote the open and zariski dense subset of those points $((a_i,b_i,c_i))_{i=1}^6\in \mcl{U}(6,2)$ such that the generic set $\mcl{E}=\{(a_i,b_i)\mid 1\leq i\leq 6\}$ give rise to sufficiently general discriminantal arrangements in the sense of  C.~A.~Athanasiadis (Definition $2.1$ in~\cite{MR1720104}). It is an open subset using Proposition 2.2 in~\cite{MR1720104}. Let $\mcl{G}(6,2)=\{((a_i,b_i))_{i=1}^6\in (\mbb{R}^2)^6\mid \Det\mattwo {a_i}{b_i}{a_j}{b_j}\neq 0, 1\leq i<j\leq 6\}$. It is an open subset in $(\mbb{R}^2)^6$.

Consider the projection $\gp:(\mbb{R}^3)^6\lra (\mbb{R}^2)^6$ which is an open map given by $((a_i,b_i,c_i))_{i=1}^6 \lra ((a_i,b_i))_{i=1}^6$. Then we have a surjection $\gp:\mcl{U}(6,2) \lra \mcl{G}(6,2)$. Now $V=\gp(\mcl{V}(6,2))$ is an open set in $\mcl{G}(6,2)$.

Let $\mcl{M}=\{v_i=(a_i,b_i)\mid 1\leq i\leq 6\}$ be a generic set which gives rise to three pairs of mutually perpendicular lines. Suppose  $((a_i,b_i))_{i=1}^6\in V$. Assume, using Corollary~\ref{cor:Acute} that, we have after renumbering the subscripts, an acute-angled triangle $\Gd L_1L_3L_5$ and $L_4\perp L_1,L_5 \perp L_2,L_6\perp L_3$  where the direction cosines of $L_i$ is either $\frac{1}{\sqrt{a_i^2+b_i^2}}(a_i,b_i)$ or $-\frac{1}{\sqrt{a_i^2+b_i^2}}(a_i,b_i),1\leq i\leq 6$. If $((a_i,b_i))_{i=1}^6$ gives rise to Figure~\ref{fig:One} which is in the set $\gp^{-1}(((a_i,b_i))_{i=1}^6)\subs \gp^{-1}(V)$, then after a small perturbation of $((a_i,b_i))_{i=1}^6$ in $V$ to $((a'_i,b'_i))_{i=1}^6$, that is, there exists an open set $W\subs V,((a_i,b_i))_{i=1}^6\in W,((a'_i,b'_i))_{i=1}^6\in W$ such that $((a'_i,b'_i))_{i=1}^6$ gives rise to Figure~\ref{fig:Two} in $\gp^{-1}(((a'_i,b'_i))_{i=1}^6)$ $\subs \gp^{-1}(W)$  where the orders of points on all the lines in Figure~\ref{fig:Two} are represented by those of some element in $\gp^{-1}(((a_i,b_i))_{i=1}^6)$. Now this is a contradiction to Proposition~\ref{prop:OrderofPoints}. Hence main Theorem~\ref{theorem:SixLines} follows.    
\end{proof}
\section{\bf{Second Proof of the First Main Theorem}}
Now as mentioned in Remark~\ref{remark:FirstTheorem}, we use the full strength of C.~A.~Athanasiadis~\cite{MR1720104} result in the second proof of Theorem~\ref{theorem:SixLines}.
\begin{proof}
If we have a ``sufficiently general" discriminantal arrangement $\mcl{C}^n_{\binom{n}{m+1}}$ then the hyperplane arrangements that arise from it have the following nice property. Let $\mcl{H}^m_n$ be any hyperplane arrangement that arises from $\mcl{C}^n_{\binom{n}{m+1}}$. Let $\mcl{S}=\{S_1,S_2,\ldots,S_r\}$ be the sets of concurrencies that exist in $\mcl{H}^m_n$. Then we have that $\mcl{S}\in P(n,m)$, that is,
\equ{\mid \us{i\in I}{\bigcup}S_i \mid>m+\us{i\in I}{\sum}(\mid S_i\mid-m)}
for all $I\subseteq \{1,2,\ldots,r\}$ with $\mid I\mid \geq 2$.

Now in the case of a line arrangement of six lines in the plane which form a triangle with three altitudes, there are four sets of concurrencies where three lines concur at each of these four points. If $S_i,i=1,2,3,4$ are the sets of concurrencies then we have $\mid \us{i=1}{\os{4}{\cup}} S_i\mid =6$ and $2+\us{i=1}{\os{4}{\sum}}(\mid S_i\mid -2)=2+(3-2)+(3-2)+(3-2)+(3-2)=6$. Hence such a configuration cannot arise from a ``sufficiently general" discriminantal arrangement. This proves Theorem~\ref{theorem:SixLines}.
\end{proof}
\begin{remark}
The same conclusion as in the second proof of Theorem~\ref{theorem:SixLines} can be obtained for six lines forming a quadrilateral and two diagonals. In fact there is a projective transformation which takes a line arrangement of six lines which form a triangle with the three altitudes to a line arrangement of six lines which form a quadrilateral and two diagonals. Also refer to Theorem $8$, page 163 in  H.~H.~Crapo~\cite{MR0843374}.   
\end{remark}
\section{\bf{Computational Verification}}
In this section we computationally verify and cohere with the fact that indeed a generic six-line arrangement, which consists of three pairs of mutually perpendicular lines, does not give rise to a very generic discriminantal arrangement with two cases of examples. For computing the number of convex cones of a discriminantal arrangement, we compute its characteristic polynomial. The method of computing the characteristic polynomial for hyperplane arrangements is a well established method. 
Articles by T.~Zaslavky~\cite{MR0400066},~\cite{MR0357135}, F.~Ardila~\cite{MR2318445}, E.~Katz~\cite{MR3702317}, and books by 
A.~Dimca ~\cite{MR3618796}, P.~Orlik \& H.~Terao~\cite{MR1217488}, R.~Stanley~\cite{MR2868112}, in which this concept is explained, are relevant.

Assume without loss of generality that, $L_1$ is the X-axis, $L_4$ is the Y-axis and $L_i$ has slope $m_i$ for $i=2,3,5,6$ with $m_5=-\frac{1}{m_2}>m_6=-\frac{1}{m_3}>m_1=0>m_2>m_3$. We consider two cases: $m_2=\frac 1{m_3}$ and $m_2\neq \frac 1{m_3}$. 
\subsection{The Case $m_2=\frac 1{m_3}$}
We mention an example here.
\begin{example}
	\label{Example:SixLines}
	~\\
	Consider the following generic six-line arrangement $\mcl{L}_6^2=\{L_1,L_2,L_3,L_4,L_5,L_6\}$ given by 
	\equa{&L_1: x_1=0, L_2: 2x_1+3x_2=-2,L_3:3x_1+2x_2=3,\\
		&L_4:x_2=0, L_5: 3x_1-2x_2=5,L_6: 2x_1-3x_2=5.}
	We have $L_1\perp L_4,L_2\perp L_5,L_3\perp L_6$. 
	The $\binom{6}{3}=20$ hyperplanes of its discriminantal arrangement in $\mbb{R}^6$ is given by 
	\equa{M_{\{1<2<3\}}&=\{(y_1,y_2,y_3,y_4,y_5,y_6)\in \mbb{R}^6\mid -5y_1-2y_2+3y_3=0\},\\ 
		M_{\{1<2<4\}}&=\{(y_1,y_2,y_3,y_4,y_5,y_6)\in \mbb{R}^6\mid 2y_1-y_2+3y_4=0\},\\
		M_{\{1<2<5\}}&=\{(y_1,y_2,y_3,y_4,y_5,y_6)\in \mbb{R}^6\mid -13y_1+2y_2+3y_5=0\},\\
		M_{\{1<2<6\}}&=\{(y_1,y_2,y_3,y_4,y_5,y_6)\in \mbb{R}^6\mid -12y_1+3y_2+3y_6=0\},\\
		M_{\{1<3<4\}}&=\{(y_1,y_2,y_3,y_4,y_5,y_6)\in \mbb{R}^6\mid 3y_1-y_3+2y_4=0\},\\
		M_{\{1<3<5\}}&=\{(y_1,y_2,y_3,y_4,y_5,y_6)\in \mbb{R}^6\mid -12y_1+2y_3+2y_5=0\},\\
		M_{\{1<3<6\}}&=\{(y_1,y_2,y_3,y_4,y_5,y_6)\in \mbb{R}^6\mid -13y_1+3y_3+2y_6=0\},\\
		M_{\{1<4<5\}}&=\{(y_1,y_2,y_3,y_4,y_5,y_6)\in \mbb{R}^6\mid -3y_1+2y_4+y_5=0\},\\
		M_{\{1<4<6\}}&=\{(y_1,y_2,y_3,y_4,y_5,y_6)\in \mbb{R}^6\mid -2y_1+3y_4+y_6=0\},\\
		M_{\{1<5<6\}}&=\{(y_1,y_2,y_3,y_4,y_5,y_6)\in \mbb{R}^6\mid -5y_1+3y_5-2y_6=0\},}
		\equa{M_{\{2<3<4\}}&=\{(y_1,y_2,y_3,y_4,y_5,y_6)\in \mbb{R}^6\mid 3y_2-2y_3-5y_4=0\},\\
		M_{\{2<3<5\}}&=\{(y_1,y_2,y_3,y_4,y_5,y_6)\in \mbb{R}^6\mid -12y_2+13y_3-5y_5=0\},\\
		M_{\{2<3<6\}}&=\{(y_1,y_2,y_3,y_4,y_5,y_6)\in \mbb{R}^6\mid -13y_2+12y_3-5y_6=0\},\\
		M_{\{2<4<5\}}&=\{(y_1,y_2,y_3,y_4,y_5,y_6)\in \mbb{R}^6\mid -3y_2+13y_4+2y_5=0\},\\
		M_{\{2<4<6\}}&=\{(y_1,y_2,y_3,y_4,y_5,y_6)\in \mbb{R}^6\mid -2y_2+12y_4+2y_6=0\},\\
		M_{\{2<5<6\}}&=\{(y_1,y_2,y_3,y_4,y_5,y_6)\in \mbb{R}^6\mid -5y_2+12y_5-13y_6=0\},}
		\equa{M_{\{3<4<5\}}&=\{(y_1,y_2,y_3,y_4,y_5,y_6)\in \mbb{R}^6\mid -3y_3+12y_4+3y_5=0\},\\
		M_{\{3<4<6\}}&=\{(y_1,y_2,y_3,y_4,y_5,y_6)\in \mbb{R}^6\mid -2y_3+13y_4+3y_6=0\},\\
		M_{\{3<5<6\}}&=\{(y_1,y_2,y_3,y_4,y_5,y_6)\in \mbb{R}^6\mid -5y_3+13y_5-12y_6=0\},\\
		M_{\{4<5<6\}}&=\{(y_1,y_2,y_3,y_4,y_5,y_6)\in \mbb{R}^6\mid -5y_4+2y_5-3y_6=0\}}
	and 
	$\mcl{C}^6_{\binom{6}{3}}=\{M_{\{i<j<k\}}\mid 1\leq i<j<k\leq 6\}$. 
	Now we use the following formula for the number of convex cones of a central arrangement.
	For a central arrangement $\mcl{A}$ in $n$-dimensional Euclidean space, the number of convex cones in the arrangement is given by 
	\equan{Regions}{(-1)^n\gch_{\mcl{A}}(-1)=(-1)^n\us{\mcl{B}\subseteq \mcl{A}}{\sum}(-1)^{\#(\mcl{B})+n-rank(\mcl{B})}.}
	Here $\gch_{\mcl{A}}(x)$ is the characteristic polynomial of $\mcl{A}$. Formula~\ref{Eq:Regions} can be derived from Proposition $3.11.3$ on Page $283$ of R.~Stanley~\cite{MR2868112}. We take here $\mcl{A}=\mcl{C}^6_{\binom{6}{3}}=B$ to simplify notation. Here $B$ denotes the $(20\times 6)$ matrix of coefficients of the variables $y_1,\ldots,y_6$ in the hyperplanes $M_{\{i<j<k\}}$, $1\leq i<j<k\leq 6$. Then we can compute the ranks of all submatrices corresponding to any finite subset $\mcl{B}$ of the rows of $B=\mcl{A}$, and use the above formula to compute $r(\mcl{C}^6_{\binom{6}{3}})=(-1)^6\gch{\mcl{C}^6_{\binom{6}{3}}}(-1)$. Upon computation we obtain that there are \equ{r(\mcl{C}^6_{\binom{6}{3}})=(-1)^6\gch_{\mcl{C}^6_{\binom{6}{3}}}(-1)=884 \text{ convex cones}.}
	
	In this example we have the slopes $m_i$ of the lines $L_i,1\leq i\leq 6$ are given by $m_1=0,m_2=-\frac 23,m_3=-\frac 32,m_4=\infty,m_5=\frac32,m_6=\frac23$ with $m_2=\frac 1{m_3}$. 
	
	Now the characteristic polynomial of a very generic discriminantal arrangement corresponding is given by 
	\equa{\gch^{\text{Very Generic}}(x)&=x^6-20x^5+145x^4-426x^3+300x^2\\
		&=x^2(x-1)(x^3-19x^2+126x-300).}
 For this we may refer to Y.~Numata, A.~Takemura~\cite{MR2986882} and H.~Koizumi, Y.~Numata, A.~Takemura~\cite{MR3000446}.
	Hence we have \equ{r^{\text{Very Generic}}=(-1)^6\gch^{\text{Very Generic}}(-1)=892\neq 884.}
	
\end{example}
\subsection{The Case $m_2\neq \frac 1{m_3}$}
We mention another example here.
\begin{example}
If we consider six lines $L_i,1\leq i\leq 6$ with slopes $m_i,1\leq i\leq 6$ respectively such that $m_1=0,m_2=-1,m_3=-2,m_4=\infty,m_5=1,m_6=\frac 12$ then we have $m_2\neq \frac 1{m_3}$ and a similar computation as in the previous case yields $888\neq 892$ convex cones of its associated discriminantal arrangement. 
\end{example}
\section{\bf{Proof of the Second Main Theorem}}
We begin with a proposition.
\begin{prop}
	\label{prop:SixLines}
	Consider the generic line arrangement given in Figure~\ref{fig:Three}.
	Here $L_1$ is the X-axis, $L_4$ is the Y-axis. The equations for $L_1,L_2,L_3,L_4,L_5,L_6$ are given as follows.
	\equa{&L_1: y=0, L_2: x-2y=7, L_3: -2x+y=4,\\
		&L_4: x=0, L_5: 5x+ y=2, L_6: x+y=-3.}
	\begin{figure}[h]
		\centering
		\includegraphics[width = 1.0\textwidth]{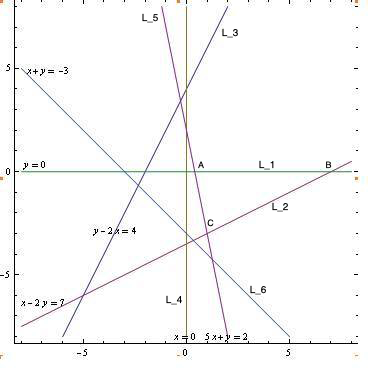}
		\caption{A Generic Six-Line Arrangement with Triangular Cell $\Gd ABC$}
		\label{fig:Three}
	\end{figure}
	Let $A=L_1\cap L_5, B=L_1\cap L_2, C=L_2\cap L_5$.
	By translating the lines $L_i, 1\leq i\leq 6$ in the plane and keeping the orders of the remaining points of intersections on each of the lines unchanged, the interchange of points $A,B$ on the line $L_1$, $B,C$ on the line $L_2$ and $A,C$ on the line $L_5$ cannot be done, that is, the orientation of the triangle $\Gd ABC$ cannot be flipped. Equivalently the point $C$ alone cannot be pushed to the first quadrant by the translations of the lines $L_2,L_3,L_5,L_6$ and keeping the orders of intersections of points the same on all the lines $L_i, 1\leq i\leq 6$ except for the interchange of points $A,B$ on the line $L_1$, $B,C$ on the line $L_2$ and $A,C$ on the line $L_5$. 
\end{prop}
\begin{proof}
	First of all we need not need to translate all the lines to prove the proposition. We fix the lines $L_1,L_4$ as $X$ and $Y$-axes respectively.
	Upon translations of the lines $L_2,L_3,L_5,L_6$, let the equations of the lines be given by 
	\equa{&L_1: y=0, L_2: x-2y=a, L_3: -2x+y=b,\\
		&L_4: x=0, L_5: 5x+ y=c, L_6: x+y=d.}
	Now we must have that $L_3\cap L_5$ is in quadrant II, $L_3\cap L_6, L_2\cap L_3$ is in quadrant III and $L_2\cap L_6,L_5\cap L_6$ is in quadrant IV. We need to prove that $C=L_2\cap L_5$ cannot be in quadrant I. Also we must have by observing the X-intercept and Y-intercept, $a>0,b>0,c>0,d<0$. By observing the order of points on the Y-axis we have $-\frac a2<d<0<c<b$. By observing the order of points on the X-axis we have $d<-\frac b2<0$. This implies $0<c<b<-2d<a$. 
	Now $C=L_2\cap L_5=(\frac{a+2c}{11},\frac{c-5a}{11})$ and here we have $\frac{a+2c}{11}>0$ and $\frac{c-5a}{11}<\frac{c-a}{11}<0$. Hence the point $C$ must lie in the quadrant IV and not in quadrant I. This proves the proposition.
\end{proof}
Now we prove the second main theorem.
\begin{proof}[Proof of Theorem~\ref{theorem:SixLinesII}]
	If the set $\mcl{U}=\{v_1=(0,1),v_2=(1,-2),v_3=(-2,1),v_4=(1,0),v_5=(5,1),v_6=(1,1)\}$ given by Figure~\ref{fig:Three} is a very generic set then we are done. This is because the hyperplane $M_{\{1,2,5\}}$ of the associated discriminantal arrangement $\mcl{C}^6_{\binom{6}{3}}$ is not a codimension-one boundary hyperplane of the cone $C$ containing the point $(b_1,b_2,b_3,b_4,b_5,b_6)=(0,7,4,0,2,-3)$ using Proposition~\ref{prop:SixLines}. If the set $\mcl{U}$ is not very generic then we use a density argument.
	The space of sufficiently general discriminantal arrangements $\mcl{C}^n_{\binom{n}{m+1}}$ is dense in the space of discriminantal arrangements given by generic finite subsets of $\mbb{R}^m$ of cardinality $n$. (Here $n=6,m=2,\binom{n}{m+1}=\binom{6}{3}=20$).  The space of very generic (ordered) sets in $(\mbb{R}^2)^n$ is dense in the space of generic (ordered) sets in $(\mbb{R}^2)^n$. There exists a small neighbourhood $V$ of the point $\mcl{U}=(v_1,\ldots,v_6)$ in $(\mbb{R}^2)^6$ which contains a very generic point $\mcl{V}=(w_1,\ldots,w_6)\in (\mbb{R}^2)^6, w_i=(a_{i1},a_{i2}),\ 1\leq i\leq 6$ which gives rise to a very generic six-line arrangement $\{L_i:a_{i1}x_1+a_{i2}x_2=c_i\mid 1\leq i\leq 6\}$ with the constant coefficient vector $(c_1,\ldots,c_6)$ in a small neighbourhood of $(b_1,\ldots,b_6)$. For this very generic discriminantal arrangement given by $\mcl{V}$, the hyperplane $M_{\{1,2,5\}}$ is not a codimension-one boundary hyperplane of the cone $C$ containing the point $(c_1,\ldots,c_6)$.  This follows since it is an open condition. This proves the second main theorem.
\end{proof}
We finally mention a remark before the appendix section.
\begin{remark}
I believe that most of the discriminantal arrangements $\mcl{C}^n_{\binom{n}{m+1}}$ which are ``sufficiently general" satisfy the property that there is a convex cone $C$ such that codimension-one boundary hyperplanes of $\mbb{R}^n$ of the convex cone $C$  containing $(c_1,c_2,\ldots,c_n)$ need not bijectively give rise to simplex cells of the hyperplane arrangement $\mcl{H}^m_n$ corresponding to $(c_1,c_2,\cdots,c_n)$. However it is still an interesting exercise to characterize ``sufficiently general" discriminantal arrangements where a correspondence exists. 
\end{remark}

\section{\bf{Appendix}}

In the appendix section we discuss the intersection lattice of a (Zariski open and dense) class of ``very generic" or ``sufficiently general" discriminantal arrangements and give a combinatorial description of the lattice. 

\subsection{\bf{The matroid of circuits of the configuration of $n$ ``generic" points in $\mbb{R}^k$}}
First we mention a remark concerning the "generic" condition.
\begin{remark}
	A ``generic" set $\mcl{O}\subs \mcl{X}$ in an affine or quasi-affine irreducible algebraic set $\mcl{X}\subseteq \mbb{R}^k$ is a nonempty Zariski open subset given by non-vanishing of certain finite collection of polynomials in $k$ variables. Hence such an open set $\mcl{O}$ is dense in $\mcl{X}$.
\end{remark}
\begin{defn}
	A matroid $M$ is a ordered pair $(E,\mcl{I})$ consisting of a finite set $E$ and collection $\mcl{I}$ of subsets of $E$ having the following properties:
	\begin{enumerate}
		\item $\es\in \mcl{I}$.
		\item If $I\in \mcl{I},I'\subseteq I$ then $I'\in \mcl{I}$.
		\item If $I_1$ and $I_2$ are in $\mcl{I}$ and $\mid I_1\mid <\mid I_2\mid$, then there exists an element $e\in I_2\bs I_1$ such that $I_1\cup\{e\}\in \mcl{I}$.
	\end{enumerate}
	The members of $\mcl{I}$ are called independent sets of $M$. A subset of $E$ that is not in $\mcl{I}$ is called dependent. A minimal dependent set is called a circuit of $M$. The set of all circuits of $M$ is denoted by $\mcl{C}(M)$. A circuit of $M$ having $n$ elements is called an $n$-circuit. A maximal independent set is called a basis of $M$. It follows that any two bases of $M$ has the same cardinality. The set of all bases of $M$ is denoted by $\mcl{B}(M)$.
\end{defn}
\begin{example}
	Let $V=\{v_i=(a_{i1},a_{i2},\cdots,a_{ik})\mid 1\leq i\leq n\}\subs \mbb{R}^k$ be a set of $n$-vectors. Let $\mcl{I}$ be the collection of the linearly independent subsets of $E$. Then $M[V]=M[A]=(V,\mcl{I})$ is a matroid where $A=[a_{ij}]_{1\leq i\leq n,1\leq j\leq k}$. It is called the vector matroid.
\end{example}
\begin{example}
	Let $E=\{1,2,\cdots,n\}$ for a positive integer $n$. Let $\mcl{I}=\{X\subseteq E \text{ such that } \mid X\mid \leq k\}$. Then $G=(E,\mcl{I})$ is a matroid. The matroid $G$ is called uniform matroid. It is also denoted by $U_{k,n}$. The collection of circuits of $U_{k,n}$ is given by $\mcl{C}(U_{k,n})=\{X\subseteq E \text{ such that } \mid X\mid= k+1\}$. The collection of bases of $U_{k,n}$ is given by $\mcl{B}(U_{k,n})=\{X\subseteq E \text{ such that } \mid X\mid= k\}$.
\end{example}
\begin{remark}
	For a set $V=\{v_1,v_2,\cdots,v_n\}$ of $n$-vectors in $\mbb{R}^k$ if, in addition, any $k\leq n$ elements in $V$ are linearly independent then $M[V]$ is a vector matroid realizing the uniform matroid $U_{k,n}$. We say that the matroid $U_{k,n}$ is representable. Sometimes when there is no ambiguity, the matroid $M[V]$ is also denoted by $U_{k,n}$. In the collection of all possible ordered $n$-tuples of vectors in $\mbb{R}^k$, that is, in the space $(\mbb{R}^k)^n$, the subset of all ordered $n$-tuples of vectors such that, any $k\leq n$ of each $n$-tuple is linearly independent, is a ``generic" set as it is given by non-vanishing of certain determinant polynomials.  
\end{remark}
Now we form another matroid on the collection $\mcl{C}(U_{k,n})$ of circuits of the matroid $U_{k,n}$.
\begin{defn}[The matroid $D(U_{k,n}), D$ stands for Dilworth]
	The Dilworth matroid is given by $D(U_{k,n})=(\mcl{C}(U_{k,n}),\mcl{D})$ where $\mcl{D}$ is defined as follows. We say a collection $\mcl{S}=\{C_1,C_2,\cdots,C_m\}$ of circuits of $U_{k,n}$ is independent in $D(U_{k,n})$ , that is, $\mcl{S}\in \mcl{D}$ if for any non-empty subset $J\subseteq \{1,2,\cdots,m\}$ we have $\mid \us{j\in J}{\cup}C_j\mid \geq k+\mid J\mid$. The collection $\mcl{B}(D(U_{k,n}))\subs \mcl{D}$ of bases of $D(U_{k,n})$ is precisely those elements of $\mcl{D}$ which have cardinality $n-k$. 
\end{defn}
\subsubsection{\bf{Construction of a Representation of the Dilworth Matroid $D(U_{k,n})$}}
Now we construct, from a representation $M[V]$ of the uniform matroid $U_{k,n}$, a representation of the Dilworth matroid $D(U_{k,n})$ when $V$ is restricted in a ``further generic" subset of $(\mbb{R}^k)^n$. This will occupy a few pages below.

Let $V=\{v_i=(a_{i1},a_{i2},\cdots,a_{ik})\mid 1\leq i\leq n\}\subs \mbb{R}^k$ be denoted by the $(n\times k)$-matrix  $A=[a_{ij}]_{1\leq i\leq n,1\leq j\leq k}$ with all the $k\times k$ minors non-zero. Let $S=\{i_1<i_2<\cdots<i_k<i_{k+1}\}\subs \{1,2,\cdots,n\}$ be a $(k+1)$-subset. Let $X=[x_{ij}]_{1\leq i\leq n,1\leq j\leq k}$ be a $n\times k$ matrix of indeterminates and $P_S(X)=\Det([x_{i_tj}]_{1\leq t,j\leq k})$. The ``generic" subset to which $V$ belongs right now is given by 
\equa{\mcl{O}=\{A&=[a_{ij}]_{1\leq i\leq n,1\leq j\leq k}\mid P_S(A)=\Det([a_{i_tj}]_{1\leq t,j\leq k})\neq 0\\& \text{for every }1\leq i_1<i_2<\cdots<i_k\leq n\}\subs \mbb{R}^{kn}.}
Here $\mcl{O}$ is non-empty Zariski open and hence a dense subset of $\mbb{R}^{kn}$ as $\mbb{R}^{kn}$ is irreducible in Zariski topology.

Let us order the collection of ascendingly sorted $(k+1)$-subsets of $\{1,2,\cdots,n\}$ in dictionary order. For example if $n=4,k=2$ then 
\equ{\{1,2,3\}<\{1,2,4\}<\{1,3,4\}<\{2,3,4\}.}
Let $y_1,y_2,\cdots,y_n$ be $n$ distinct indeterminates. For every $(k+1)$-subset $S=\{i_1<i_2<\cdots<i_k<i_{k+1}\}\subs \{1,2,\cdots,n\}$ consider the following linear polynomial in the $n$ variables $y_1,y_2,\cdots,y_n$. 
\equ{L_{S,X=[x_{ij}]_{1\leq i\leq n,1\leq j\leq k}}(y_1,y_2,\cdots,y_n)=\Det\begin{pmatrix}
		x_{i_11} & x_{i_12} & x_{i_13} & \cdots & x_{i_1k} & y_{i_1}\\
		x_{i_21} & x_{i_22} & x_{i_23} & \cdots & x_{i_2k} & y_{i_2}\\
		x_{i_31} & x_{i_32} & x_{i_33} & \cdots & x_{i_3k} & y_{i_3}\\
		\vdots	 & \vdots   & \vdots   & \ddots & \vdots   & \vdots\\
		x_{i_k1} & x_{i_k2} & x_{i_k3} & \cdots & x_{i_kk} & y_{i_k}\\
		x_{i_{k+1}1} & x_{i_{k+1}2} & x_{i_{k+1}3} & \cdots & x_{i_{k+1}k} & y_{i_{k+1}}\\	
\end{pmatrix}}
Now form the $\binom{n}{k+1} \times n$ matrix $\Disc(A)$ using the ordering mentioned on the collection of $(k+1)$-subsets of $\{1,2,\cdots,n\}$ where the row corresponding to $(k+1)$-subset $S$ is the $n$-dimensional real coefficient vector of the linear polynomial \equ{L_{S,A=[a_{ij}]_{1\leq i\leq n,1\leq j\leq k}}(y_1,y_2,\cdots,y_n).} Note that each row of $\Disc(A)$ has $(k+1)$ nonzero entries and the remaining $n-k-1$ entries are zero. 

We will show that there exists a zariski open and hence dense subset $\mcl{U}\subs\mcl{O} \subs \mbb{R}^{kn}$ such that for $A\in \mcl{U}$ the vector matroid $M[A]$ represents the uniform matroid $U_{k,n}$ and the vector matroid $M[\Disc(A)]$ represents the Dilworth matroid $D(U_{k,n})$. This result is due to H.~Crapo~\cite{MR0843374}, Section 6, Page 149, Theorem 2.  

Let $X=[x_{ij}]_{1\leq i\leq n,1\leq j\leq k}$ be the matrix of indeterminates. Let $\Disc(X)$ be the $\binom{n}{k+1} \times n$ matrix where the row corresponding to the $(k+1)$-subset $S=\{i_1<i_2<\cdots<i_k<i_{k+1}\}\subs\{1,2,\cdots,n\}$ is the $n$-dimensional coefficient vector of the linear polynomial $L_{S,X=[x_{ij}]_{1\leq i\leq n,1\leq j\leq k}}(y_1,y_2,\cdots,y_n)$. Then we immediately observe that $\Disc(X).X=0$ the zero matrix of size ${\binom{n}{k+1}\times k}$. This follows because the determinant of a square matrix with a repeated row or a repeated column is zero. So if $A\in \mcl{O}$ then the rows of $\Disc(A)$ are orthogonal to the columns of $A$. Since $A$ has rank $k$ the rank of $\Disc(A)$ is at most $n-k$. Let $n-k<l\leq n$. So the determinant of any $l\times l$ minor in $\Disc(A)$ is zero. Since this is true for every $A\in \mcl{O}$ and $\mcl{O}$ is Zariksi dense in $\mbb{R}^{kn}$, the determinant polynomial in the variables $X=[x_{ij}]_{1\leq i\leq n,1\leq j\leq k}$ of any $l\times l$ minor of $\Disc(X)$ is an identically zero polynomial. What about $l\times l$ minors of the matrix $\Disc(X)$ for $1\leq l\leq n-k$? The following theorem answers this question to some extent. 

\begin{theorem}
	\label{theorem:Minors}
	Let $S_1,S_2,\cdots,S_l$ be $(k+1)$-subsets of $\{1,2,\cdots,n\}$ and suppose $\us{i=1}{\os{l}{\cup}} S_i=\{j_1,j_2,\ldots,j_m\}$.
	\begin{enumerate}
		\item 
		For $1\leq l < m-k$, the determinant polynomial of some $l\times l$ minor in the corresponding $l\times m$ submatrix of the $\binom{n}{k+1}\times n$ matrix $\Disc(X)$ in the variable entries $X=[x_{ij}]_{1\leq i\leq n,1\leq j\leq k}$ is not identically zero if the collection $\{S_1,S_2,\cdots,S_l\}$ is independent in the Dilworth matroid $D(U_{k,n})$. 
		\item For $l=m-k$, the determinant polynomial of any $l\times l$ minor in the corresponding $l\times m$ submatrix of the $\binom{n}{k+1}\times n$ matrix $\Disc(X)$ in the variable entries $X=[x_{ij}]_{1\leq i\leq n,1\leq j\leq k}$ is not identically zero if the collection $\{S_1,S_2,\cdots,S_l\}$ is independent in the Dilworth matroid $D(U_{k,n})$.
		\item Also the determinant polynomial of any $l\times l$ minor in the corresponding $l\times m$ submatrix of the $\binom{n}{k+1}\times n$ matrix $\Disc(X)$ in the variable entries $X=[x_{ij}]_{1\leq i\leq n,1\leq j\leq k}$ is identically zero if the collection $\{S_1,S_2,\cdots,S_l\}$ is not independent in the Dilworth matroid $D(U_{k,n})$.
	\end{enumerate}
\end{theorem}
\begin{proof}
	Because of the nature of the entries of the matrix $\Disc(X)$ and the fact that any $l\times l$ minor is zero for $l>n-k$, it is enough to prove the theorem for $l=n-k$. So we assume that $l=n-k$. Even in this case $l=n-k$, if $\{S_1,S_2,\cdots,S_{n-k}\}$ is not independent in $D(U_{k,n})$ then it is clear that the determinant polynomial of any $(n-k)\times (n-k)$ minor in the corresponding $(n-k)\times m$ submatrix of the $\binom{n}{k+1}\times n$ matrix $\Disc(X)$, in the variable entries $X=[x_{ij}]_{1\leq i\leq n,1\leq j\leq k}$ is identically zero. An example is given in~\ref{Example:Minors} illustrating this case. Also an example is given in~\ref{Example:MinorCase1} illustrating assertion Theorem~\ref{theorem:Minors}(1).
	
	So we assume that $\{S_1,S_2,\cdots,S_{n-k}\}$ is a basis of $D(U_{k,n})$ and so $m=n$.
	Here we produce a real matrix $A=[a_{ij}]_{1\leq i\leq n,1\leq j\leq k}$ such that every $k\times k$ minor is nonzero such that for the associated matrix $\Disc(A)$, in the $(n-k)\times n$ submatrix corresponding to the rows $S_1,S_2,\cdots,S_{n-k}$, every $(n-k)\times (n-k)$ minor is nonzero. 
	
	Consider $(k+1)\binom{n}{k+1}$ variables $X_{a,S}$ for $a\in S$ and any $(k+1)$-subset $S$ of $\{1,2,\cdots,n\}$. Replace the nonzero polynomial entries of the matrix $\Disc(X)$ by these variables in their exact positions. Let us denote the new matrix by $Y$.
	\begin{claim}
		Any $(n-k)\times(n-k)$ minor of the $(n-k)\times n$ submatrix of $Y$ corresponding to the rows $S_1,S_2,\cdots,S_{n-k}$ is nonzero where $\{S_1,S_2,\cdots,S_{n-k}\}$ is a basis of $D(U_{k,n})$.	
	\end{claim}
	\begin{proof}[Proof of the Claim]
		This is proved by an application of Hall's marriage theorem. Let $T=\{i_1<i_2<\cdots<i_{n-k}\}\subs\{1,2,\cdots,n\}$ and denote $T^c=\{1,2,\cdots,n\}\bs\{i_1$ $<i_2<\cdots<i_{n-k}\}$. Then we have for any subset $J\subs \{1,2,\cdots,(n-k)\},\mid\us{j\in J}{\cup}(S_j\cap T)\mid= \mid\big(\us{j\in J}{\cup}S_j\big)\cap T\mid= \mid\big(\us{j\in J}{\cup}S_j\big)\bs T^c\mid\geq \mid\us{j\in J}{\cup}S_j\mid-\mid T^c\mid\geq k+\mid J\mid-k=\mid J\mid$. Hence Hall's marriage condition is satisfied for the sets $S_i\cap T,i=1,2,\cdots,n-k$. So there exists $a_i\in S_i\cap T,1\leq i\leq n-k,a_i\neq a_j, 1\leq i\neq j\leq n-k$. Hence there is a diagonal of variables $X_{a_i,S_i}$ in the $(n-k)\times (n-k)$ in the submatrix corresponding to $S_1,S_2,\cdots,S_{n-k}$ and $T=\{i_1<i_2<\cdots<i_{n-k}\}$ which give rise to a monomial term in determinant expansion of its minor of $Y$. So the $(n-k)\times (n-k)$ minor of $Y$ is nonzero. This holds for any set $T$ of cardinality $n-k$. Hence the claim follows. 
	\end{proof} 
	Continuing with the proof of Theorem~\ref{theorem:Minors}, we can choose a set of real numbers for the entries $X_{a,S}$ in $Y$ such that any $(n-k)\times(n-k)$ minor of the $(n-k)\times n$ submatrix of $Y$ corresponding to the rows $S_1,S_2,\cdots,S_{n-k}$ is nonzero where $\{S_1,S_2,\cdots,S_{n-k}\}$ is a basis of $D(U_{k,n})$. Let us denote this $(n-k)\times n$ submatrix by $C$. Let $D$ be a $n\times k$ real matrix such that the columns of $D$ span the $k$-dimensional space in $\mbb{R}^n$ which is orthogonal to the $(n-k)$-dimensional row subspace of $C$ in $\mbb{R}^n$. Then every $k\times k$ minor of $D$ is nonzero. This follows from the claim below.
	\begin{claim}
		Let $Q$ be a non-singular matrix $n\times n$ matrix such that first $k$ rows $q_1,q_2,$ $\cdots,q_k$ span a space $W_1\subs \mbb{R}^n$ and the remaining $(n-k)$ rows $q_{k+1},q_{k+2},\cdots q_{n}$ span a space $W_2\subs \mbb{R}^n$. Assume also that $W_1 \perp W_2$, that is, $W_1$ is orthogonal to $W_2$. Then a $k\times k$ minor in the first $k$ rows is nonzero if and only if its complementary $(n-k)\times (n-k)$ minor in the remaining $n-k$ rows is nonzero. 
	\end{claim} 
	\begin{proof}[Proof of the Claim]
		We have $\dim(W_1)=k,\dim(W_2)=n-k$. Let $f_1,f_2,\cdots f_k$ be an orthonormal row basis of $W_1$ and $f_{k+1},f_{k+2},$ $\cdots,f_n$ be an orthonormal row basis of $W_2$ such that the matrix $n\times n$ matrix of rows $f_1,\cdots,f_n$ is a special orthogonal matrix $O$. Now there exists a square matrix $F$ of size $k$ and a square matrix $H$ of size $n-k$ such that
		\equ{F\matcolthree{q_1}{\vdots}{q_k}=\matcolthree{f_1}{\vdots}{f_k} \text{ and } H\matcolthree{q_{k+1}}{\vdots}{q_n}=\matcolthree{f_{k+1}}{\vdots}{f_n} \text{ where }Q=\matcolthree{q_1}{\vdots}{q_n}.}
		A $k\times k$ minor in the first $k$ rows of $Q$ is nonzero if and only if the corresponding $k\times k$ minor is nonzero in $O$.  
		A similar assertion holds for the remaining $n-k$ rows of $Q$ and $O$. Now we have reduced the claim to a special orthogonal matrix $O$.
		By using appropriate permutation matrices, it is enough to prove the claim for the principal $k\times k$ minor of $O$ and its complementary $(n-k)\times (n-k)$ minor of $O$. 
		
		Now we have if $O=\mattwo{M_1}{M_2}{M_3}{M_4}$ then $O^{-1}=O^t=\mattwo {M_1^t}{M_3^t}{M_2^t}{M_4^t}$ and $\Det(O)=\Det(O^t)=1$. So we have 
		\equ{\mattwo {M_1}{M_2}{0_{(n-k)\times k}}{I_{(n-k)\times (n-k)}}\mattwo{M_1^t}{M_3^t}{M_2^t}{M_4^t}=\mattwo{I_{k \times k}}{0_{k\times (n-k)}}{M_2^t}{M_4^t}.}
		Hence $\Det(M_1)=\Det(M_4^t)=\Det(M_4)$. This proves the claim.
	\end{proof}
	Continuing with the proof of Theorem~\ref{theorem:Minors}, we have that all $k\times k$ minors of $D$ are nonzero. Hence $M[D]$ represents the uniform matroid $U_{k,n}$. Form the matrix $\Disc(D)$ and consider the $(n-k)\times n$ submatrix $C'$ of rows of $\Disc(D)$ corresponding to the basis $\{S_1,S_2,\cdots,S_{n-k}\}$ of $D(U_{k,n})$. Now any two corresponding rows of the matrices $C'$ and $C$ are both orthogonal to the columns of $D$, that is, $C'D=0=CD$ and they both give linear dependence relations of same $k+1$ rows of $D$. Hence each row of $C'$ is proportional to corresponding row of $C$ and not all entries of these two corresponding rows are zero. Therefore every $(n-k)\times (n-k)$ minor of $C'$ is nonzero, since every $(n-k)\times (n-k)$ minor of $C$ is nonzero. The matrix $D$ is the required matrix that we were looking for.
	
	This proves that the determinant polynomial of each $(n-k)\times (n-k)$ minor of the submatrix of $\Disc(X)$ corresponding to the rows $S_1,S_2,\cdots,S_{n-k}$ is not identically zero, if $\{S_1,S_2,\cdots,S_{n-k}\}$ is a basis for $D(U_{k,n})$. Hence Theorem~\ref{theorem:Minors} follows.    
\end{proof}

So we define the Zariski dense open subset $\mcl{U}\subs \mcl{O}\subs \mbb{R}^{kn}$ as follows. Let $X$ and $\Disc(X)$ be defined as before. Let $\mcl{S}=\{S_1,S_2,\cdots,S_l\}$ be a collection of $(k+1)$-subsets of $\{1,2,\cdots,n\}$. Consider the $l\times n$ submatrix $Y_{\mcl{S}}$ of $\Disc(X)$ corresponding to the rows $S_1,S_2,\cdots, S_l$. Let $P_{\mcl{S}}(X)$ be the sum of squares of $l\times l$ minors of the matrix $Y_{\mcl{S}}$. Then we have using Theorem~\ref{theorem:Minors} that, $P_{\mcl{S}}(X)\in \mbb{R}[x_{ij}]$ is not identically zero if and only if the collection $\mcl{S}$ is an independent set in the Dilworth matroid $D(U_{k,n})$. So let 
\equan{VeryGeneric}{\mcl{U}=\{A&=[a_{ij}]_{1\leq i\leq n,1\leq j\leq k}\in \mcl{O}\subs \mbb{R}^{kn}\mid P_{\mcl{S}}(A)\neq 0\\ 
	&\text{whenever }P_{\mcl{S}}[X] \text{ is not an identically zero polynomial}\}.}

\begin{example}
	\label{Example:Minors}
	Let $n=6,k=2$ and $l=4$ with $S_1=\{1,2,3\},S_2=\{1,2,4\},S_3=\{2,3,4\},S_4=\{4,5,6\}$. Then the collection $\{S_1,S_2,S_3,S_4\}$ is not independent in $D(U_{k,n})$ because $\mid S_1\cup S_2\cup S_3\mid=4<2+3=5$. So all the $3\times 3$ minors of the submatrix corresponding to the rows $S_1,S_2,S_3$ are zero because $3>4-2=2$. Hence all the $4\times 4$ minors of the submatrix corresponding to the rows $S_1,S_2,S_3,S_4$ are zero.
\end{example}
\begin{example}
	\label{Example:MinorCase1}
	Let $n=9,k=2,l=3$ with $S_1=\{1,2,3\},S_2=\{4,5,6\},S_3=\{7,8,9\}$. Then the collection $\{S_1,S_2,S_3\}$ is independent in $D(U_{k,n})$. In the $3\times 9$ submatrix corresponding to rows $S_1,S_2,S_3$, there are $3\times 3$ minors which are nonzero as well as there are $3\times 3$ minors which are identically zero in the variables $[x_{ij}]_{1\leq i\leq 9,1\leq j\leq 2}$. 
\end{example}
So we have proved the following theorem.
\begin{theorem}
	\label{theorem:Represent}
	If $A\in \mcl{U}$ then $M[A]$ represents the uniform matroid $U_{k,n}$ and $M[\Disc(A)]$ represents the Dilworth matroid $D(U_{k,n})$.
\end{theorem}
\subsection{\bf{The Intersection Lattice of a Very Generic Discriminantal Arrangement}}
Now we define a very generic hyperplane arrangement and a very generic discriminantal arrangement, though it is mentioned in Definition~\ref{defn:GenericVeryGeneric}.
\begin{defn}
	\label{defn:GenericVeryGenericArrangements}
	Let $k,n$ be positive integers. Let \equ{\mcl{H}_n^k=\{H_1,H_2,\ldots,H_n\}} 
	be a generic hyperplane arrangement of $n$ hyperplanes in $\mbb{R}^k$. Let the equation for $H_i$ be given by 
	\equ{\us{j=1}{\os{k}{\sum}}a_{ij}x_j=b_i,\text{ with } a_{ij}, b_i\in \mbb{R}, 1\leq j\leq k, 1\leq i\leq n.}
	Let $\mcl{U}$ be the zarski dense open set in $\mbb{R}^{kn}$ as mentioned in Theorem~\ref{theorem:Represent} and in Equation~\ref{Eq:VeryGeneric}. 
	We say the generic hyperplane arrangement $\mcl{H}_n^k$ is very generic if $A=[a_{ij}]_{1\leq i\leq n,1\leq j\leq k}\in \mcl{U}$.
	Also here the associated discriminantal arrangement of hyperplanes passing through the origin in $\mbb{R}^n$ 
	is given by
	\equ{\mcl{C}^n_{\binom{n}{k+1}}=\{M_{\{i_1,i_2,\ldots,i_{k+1}\}}\mid 1\leq 
		i_1<i_2<\ldots<i_{k+1}\leq n\}} where the hyperplane $M_{\{i_1,i_2,\ldots,i_{k+1}\}}$ 
	passing through the origin in $\mbb{R}^n$ in the variables $y_1,y_2,\ldots,y_n$ has the equation given 
	by 
	\equ{\Det
		\begin{pmatrix}
			a_{i_11} & a_{i_12} & \cdots & a_{i_1(k-1)} & a_{i_1k} & y_{i_1}\\
			a_{i_21} & a_{i_22} & \cdots & a_{i_2(k-1)} & a_{i_2k} & y_{i_2}\\
			\vdots   & \vdots   & \ddots & \vdots       & \vdots   & \vdots\\
			a_{i_{k-1}1} & a_{i_{k-1}2} & \cdots & a_{i_{k-1}(k-1)} & a_{i_{k-1}k} & y_{i_{k-1}}\\
			a_{i_k1} & a_{i_k2} & \cdots & a_{i_k(k-1)} & a_{i_mm} & y_{i_k}\\
			a_{i_{k+1}1} & a_{i_{k+1}2} & \cdots & a_{i_{k+1}(k-1)} & a_{i_{k+1}k} & y_{i_{k+1}}\\
		\end{pmatrix}
		=0}
	is said to be a very generic discriminantal arrangement.
\end{defn}
Now we describe a lattice which arises from the Dilwork matroid $D(U_{k,n})$.
\subsubsection{\bf{Dilworth Lattice $P(n,k)$}}
\begin{defn}
	The set $P(n,k)$ consists of collections of all sets of the form $\mcl{S}=\{S_1,S_2,\cdots,S_r\}$, where $S_i\subs\{1,2,\cdots,n\}$ where $S_i\subs\{1,2,\cdots,n\}$, each of cardinality at least $k+1$ such that $\mid \us{i\in I}{\cup} S_i\mid > k+ \us{i\in I}{\sum} (\mid S_i\mid-k)$ for all $I\subs \{1,2,\cdots,r\}$ with $\mid I \mid \geq 2$. They partially order $P(n,k)$ by letting $\{S_1,S_2,\cdots,S_r\}=\mcl{S}\leq \mcl{T}=\{T_1,T_2,\cdots,T_p\}$, if, for each $1\leq i\leq r$ there exists $1\leq j\leq p$ such that $S_i\subseteq T_j$.
\end{defn}
Now we show that this lattice is isomorphic to the intersection lattice of a very generic discriminantal arrangement. Let $A\in \mcl{U}$ which represents the uniform matroid $U_{k,n}$ and $M[\Disc(A)]$ which represents the Dilworth matroid $D(U_{k,n})$ be such that $A$ gives rise to a very generic discriminantal arrangement $\mcl{C}^n_{\binom{n}{k+1}}$. Now the intesection lattice of $\mcl{C}^n_{\binom{n}{k+1}}$ is isomorphic to the lattice of flats $L(n,k)$ of that arises from the rows of $\Disc(A)$ by taking orthogonal complement.
The lattice $L(n,k)$ is the lattice of subspaces of $\mbb{R}^n$ which are spanned by the rows of $\binom{n}{k+1}\times n$ matrix $\Disc(A)$.  For any collection $\{S_1,S_2,\cdots,S_m\}$ of $(k+1)$-subsets we have 
\equ{\us{i=1}{\os{m}{\cap}}M_{S_i}= \langle \ga_{S_i}:1\leq i\leq m\rangle^{\perp}}
where $\ga_{S_i},1\leq i\leq m$ are the rows in $\Disc(A)$ corresponding to rows $S_1,S_2,\cdots,S_m$.

For an arbitrary antichain $\mcl{S}=\{S_1,S_2,\cdots,S_m\}$ of subsets of $\{1,2,\cdots,n\}$ define $V_{\mcl{S}}$ as the row span of those rows in $\Disc(A)$ which correspond to sets $S$ of cardinality $k+1$ such that $S\subs S_i$ for some $i, 1\leq i\leq m$. If $\mcl{S}=\{S\}$ then denote $V_{\mcl{S}}$ by $V_{S}$. So if $\mid S\mid\geq k+1$ then $\dim V_{S}=\mid S\mid-k$. If $\mcl{S}\in P(n,k)$ and $S_1,S_2\in \mcl{S}$ are two distinct elements then $\mid S_1\cap S_2\mid<k$. So here $\mcl{S}$ is an antichain.
Now we define a map $\gf: P(n,k)\lra L(n,k)$ as $\gf(\mcl{S})=V_{\mcl{S}}$.

\begin{theorem}
	\label{theorem:IsoPoset}
	The map $\gf: P(n,k)\lra L(n,k)$ is an isomorphism of posets.
\end{theorem} 

We give a proof of this theorem after proving the following three lemmas.
\begin{defn}
	Let $\gn(S)=\max(0,\mid S\mid-k)$. For $\mcl{F}=\{S_1,S_2,\cdots,S_m\}$ a collection of subsets of $\{1,2,\cdots,n\}$ define
	\equ{\Gd(\mcl{F})=\gn(\us{i=1}{\os{m}{\cup}}S_i)-\us{i=1}{\os{m}{\sum}}\gn(S_i).}
\end{defn}
\begin{remark}
	If $\mcl{S}=\{S_1,S_2,\ldots,S_m\}$ is a collection of $(k+1)$-subsets independent in $U_{k,n}$   then $\Gd(\mcl{F})\geq 0$ for all $\mcl{F}\subseteq \mcl{S}$.
	If $\mcl{S}\in P(n,k)$ then $\Gd(\mcl{F})>0$ whenever $\mcl{F}\subseteq \mcl{S}$ and it has at least two elements.
\end{remark}
\begin{lemma}
	\label{lemma:Antichain}
	Let $\mcl{S}$ be an antichain with the properties $\mid S\mid \geq k+1$ for all $S\in \mcl{S}$ and $\Gd(\mcl{F})\geq 0$ for all $\mcl{F}\subseteq \mcl{S}$. Suppose we have that $V_\mcl{S}=\us{S\in \mcl{S}}{\sum}V_S=\us{S\in\mcl{S}}{\bigoplus}V_S$, that is, the sum is direct. If $S\nin P(n,k)$ then there exists an antichain $\mcl{S}'$ with the same three properties as $\mcl{S}$ and such that $V_{\mcl{S}}=V_{\mcl{S}'}, \mid \mcl{S}\mid >\mid \mcl{S}'\mid$ and $\mcl{S}<\mcl{S}'$.   	
\end{lemma}
\begin{proof}
	By assumption, $\Gd(\mcl{F})=0$ for some $\mcl{F}\subs \mcl{S}$ with at least two elements. Let $U=\cup \mcl{F}$. Now we observe that 
	$\us{S\in \mcl{F}}{\bigoplus} V_{S}=\us{S\in \mcl{F}}{\sum} V_{S}\subseteq V_{U}$ and $\us{S\in \mcl{F}}{\sum}\dim V_{S}=\us{S\in \mcl{F}}{\sum}\gn(S)=\gn(U)=\dim V_U$. Hence we have $\us{S\in \mcl{F}}{\bigoplus} V_{S}=\us{S\in \mcl{F}}{\sum} V_{S}=V_{U}$. Replace the sets $S\in \mcl{F}\subseteq \mcl{S}$ with their union $U$ to get a collection $\mcl{S}'$. Now $\mcl{S}'$ is antichain. Suppose not, that is, there exists an element $S\in \mcl{S}\bs \mcl{F}$ such that $S\subseteq U =\cup \mcl{F}$. Then we have 
	$0\leq \Gd(\cup \mcl{F}\cup S)=\gn(U)-\us{F\in \mcl{F}}{\sum}\gn(F)-\gn(S)=\Gd(\mcl{F})-\gn(S)=-\gn(S)<0$ which is a contradiction. Hence $\mcl{S}'$ is an antichain and satisfies the property that $\mcl{E}\subseteq \mcl{S}'\Ra\Gd(\mcl{E})\geq 0,\mid \mcl{S}\mid >\mid\mcl{S}'\mid$ and $\mcl{S}<\mcl{S}'$. This proves the lemma.  	
\end{proof}
\begin{lemma}
	\label{lemma:Monotone}
	If $\mcl{S}$ and $\mcl{T}$ are any antichains with $\mcl{S}\leq \mcl{T}$ and $\Gd(\mcl{F})\geq 0$ for all $\mcl{F}\subseteq \mcl{S}$, then $\us{S\in \mcl{S}}{\sum}\gn(S)\leq \us{T\in \mcl{T}}{\sum}\gn(T)$.
\end{lemma}
\begin{proof}
	The proof is by induction on the cardinality of $\mcl{S}$. Choose a set $T\in \mcl{T}$ so that the subfamily $\mcl{F}=\{F\in \mcl{S}: F\subseteq T\}$ is non-empty. Then we have 
	\equ{\us{S\in \mcl{S}}{\sum}\gn(S)=\us{F\in \mcl{F}}{\sum}\gn(F)+\us{S\in \mcl{S}\bs \mcl{F}}{\sum}\gn(S) \leq \gn(\cup\mcl{F})+\us{S\in \mcl{S}\bs \mcl{F}}{\sum}\gn(S)\leq \gn(T)+\us{S\in \mcl{S}\bs \mcl{F}}{\sum}\gn(S).}
	Since $\mcl{S}\bs\mcl{F}\leq \mcl{T}\bs\{T\}$, the lemma follows by induction.
\end{proof}
\begin{lemma}
	\label{lemma:PosetInequality}
	If $\mcl{S}\in P(n,k), T\subseteq \{1,2,\cdots,n\}$ with $\mid T\mid \geq k+1$ and suppose for all $R\subseteq T$ with $\mid R\mid =k+1$ we have $\{R\}\leq \mcl{S}$, then $\{T\}\leq \mcl{S}$. 
\end{lemma}
\begin{proof}
	Suppose $S_1,S_2\in \mcl{S}$ are two distinct sets. Then $\mid S_1\cup S_2\mid > \mid S_1 \mid + \mid S_2\mid -k$. So we get $\mid S_1\cap S_2\mid<k$. 
	If $R_1\subseteq T, R_2\subseteq T$ are two $(k+1)$-subsets such that $\mid R_1\cap R_2\mid=k$ and $R_1\subseteq S_1\in \mcl{S},R_2\subseteq S_2\in \mcl{S}$ then we have $S_1=S_2$. So $R_1\cup R_2\subseteq S_1=S_2$. By applying this procedure repeatedly we conclude that $\{T\}\leq \mcl{S}$.	
\end{proof} 

Now we prove Theorem~\ref{theorem:IsoPoset}.
\begin{proof}
	If $\mcl{S},\mcl{T}\in P(n,k),\mcl{S}\leq \mcl{T}$ the  $V_{\mcl{S}}\subseteq V_{\mcl{T}}$. So the map $\gf$ is order preserving.
	Let $V$ be a span of some rows of $\Disc(A)$, that is, $V\in L(n,k)$. Let $S_1,S_2,\cdots S_m$ be an independent set in $D(U_{k,n})$ such that the rows corresponding to them in $\Disc(A)$ span $V$, that is, if $\mcl{S}=\{S_1,S_2,\cdots,S_m\}$ then $V=V_{\mcl{S}}$. Now $\mcl{S}$ is an antichain satisfying the hypothesis of Lemma~\ref{lemma:Antichain}. If $\mcl{S}\in P(n,k)$ then $\gf(\mcl{S})=V_{\mcl{S}}$. If not, then by the repeated application of Lemma~\ref{lemma:Antichain}, we obtain an antichain $\mcl{T}$ such that $\mcl{S}\leq \mcl{T}$ such that $V=V_{\mcl{S}}=V_{\mcl{T}}$ and $\mcl{T}\in P(n,k)$ also $\mcl{T}$ satisfies the hypothesis of Lemma~\ref{lemma:Antichain}, that is, $V_\mcl{T}=\us{T\in \mcl{T}}{\sum}V_T=\us{T\in\mcl{T}}{\bigoplus}V_T$, the sum is direct and $\gf(\mcl{T})=V_{\mcl{T}}$. So the map $\gf$ is surjective.	
	
	Now we prove that $\gf$ is injective. Suppose on the contrary, we have that $\gf(\mcl{S})=\gf(\mcl{T})$ for two distinct elements $\mcl{S},\mcl{T}\in P(n,k)$. We assume without loss of generality that $\mcl{T}\leq \mcl{S}$ is not valid. By Lemma~\ref{lemma:PosetInequality} there exists a $(k+1)$-set $R$ which is contained in subset of $\mcl{T}$ but not contained in any subset of $\mcl{S}$. Since $V_{\mcl{S}}=V_{\mcl{T}}$, We can choose a minimal collection $\mcl{F}\leq \mcl{S}$ of $(k+1)$-subsets such that $\mcl{F}$ is a independent in the Dilworth lattice $D(U_{k,n})$ and the row corresponding to $R$ in $\Disc(A)$ is linearly dependent on the rows corresponding to $F\in \mcl{F}$ in $\Disc(A)$. Hence we have by minimality of $\mcl{F}$, 
	\equ{\gn(\cup \mcl{F}\cup R)<\mid \mcl{F}\mid +1 \text{ and on the other hand } \gn(\cup \mcl{F})\geq \mid \mcl{F}\mid.}
	Hence we conclude that $R\subs \cup \mcl{F}, \Gd(\mcl{F})=0$, that is, $\gn(\cup\mcl{F})=\mid \mcl{F}\mid$.
	
	Let $\mcl{E}$ be the subcollection consisting of those subsets $E\in \mcl{S}$ which contain some $F\in \mcl{F}$. Let $\mcl{E}'$ be the subcollection obtained from $\mcl{E}$ by intersecting all sets in $\mcl{E}$ with $\cup \mcl{F}$. So we have $\mcl{F}\leq \mcl{E}'$ and 
	\equ{\gn(\cup \mcl{E}')=\gn(\cup \mcl{F})=\us{F\in \mcl{F}}{\sum}\gn(F)\leq \us{E'\in \mcl{E}'}{\sum}\gn(E')}
	where the last inequality follows form Lemma~\ref{lemma:Monotone}. Now by adding further elements to each $E'\in \mcl{E}'$ we get that 
	\equ{\gn(\cup \mcl{E})\leq \us{E\in  \mcl{E}}{\sum}\gn(E) \Ra \Gd(\mcl{E})\leq 0.} By the choice of $R$ we have that $\mcl{E}$ has at least two elements. Hence we arrive at a contradiction to $\mcl{S}\in P(n,k)$. So $\gf$ is injective.
	
	Now we prove that the inverse of $\gf$ is order preserving. Let $\mcl{S},\mcl{T}\in P(n,k)$ be such that $\gf(\mcl{S})=V_{\mcl{S}}\subseteq V_{\mcl{T}}=\gf(\mcl{T})$. Choose a basis of rows in $\Disc(A)$ for $V_{\mcl{S}}$, that is, an independent set $\mcl{S}_0$ in $D(U_{k,n})$ and extend it to a basis of rows in $\Disc(A)$ for $V_{\mcl{T}}$, that is, an independent set $\mcl{T}_0$ in $D(U_{k,n})$. So we have $\mcl{S}_0\subseteq \mcl{T}_0$ are a collection of $(k+1)$-subsets. Now we apply Lemma~\ref{lemma:Antichain} repeatedly to obtain $\mcl{S}_1$ such that $\mcl{S}_0\leq \mcl{S}_1$ and $\gf(\mcl{S})=V_{\mcl{S}}=V_{\mcl{S}_0}=V_{\mcl{S}_1}=\gf(\mcl{S}_1)$. So $\mcl{S}=\mcl{S}_1$ by injectivity of $\gf$. Now the collection $\mcl{S}_1\cup(\mcl{T}_0\bs\mcl{S}_0)$ is an antichain because the sum $V_{\mcl{S}_1}+V_{\mcl{T}_0\bs\mcl{S}_0}$ is direct and all elements of $\mcl{T}_0\bs\mcl{S}_0$ are independent $(k+1)$-subsets. Also we have $\Gd(\mcl{E})\geq 0$ for all $\mcl{E}\subseteq \mcl{S}_1\cup(\mcl{T}_0\bs\mcl{S}_0)$ we have $\Gd(\mcl{E})\geq 0$. Now we apply Lemma~\ref{lemma:Antichain} to obtain a set $\mcl{T}_1$ such that $\mcl{S}_1\cup(\mcl{T}_0\bs\mcl{S}_0) \leq \mcl{T}_1$ and $\gf(\mcl{T})=V_{\mcl{T}}=V_{\mcl{T}_0}=V_{\mcl{T}_1}=\gf(\mcl{T}_1)$. So we have $\mcl{T}=\mcl{T}_1$ as $\gf$ is injective and $\mcl{S}_1\leq \mcl{T}_1\Ra \mcl{S}\leq \mcl{T}$. This proves that the inverse of $\gf$ is order preserving. Hence Theorem~\ref{theorem:IsoPoset} follows.
\end{proof}
\subsection{\bf{New Description of the Lattice Elements in $P(n,k)$}}
In this section we describe the lattice elements in a more geometric manner. Here again we assume that $A\in \mcl{U}$ which represents the uniform matroid $U_{k,n}$ and $M[\Disc(A)]$ which represents the Dilworth matroid $D(U_{k,n})$ be such that $A$ gives rise to a very generic discriminantal arrangement $\mcl{C}^n_{\binom{n}{k+1}}$. 
\begin{defn}[Concurrency Closed Sub-collection and Concurrency Closure]
	\label{defn:CC}
	~\\
	Let $n>k$ be two positive integers. Let \equ{\mcl{E}=\{\{i_1,i_2,\cdots,i_{k+1}\}\mid 1\leq i_1<i_2<\cdots<i_k<i_{k+1}\leq n\}} be the collection of all subsets of 
	cardinality $k+1$. Let $\mcl{D}\subs \mcl{E}$ be any arbitrary collection. 
	
	We say $\mcl{D}$ is concurrency closed if the following criterion for any element $S\in \mcl{E}$ is satisfied with respect to $\mcl{D}$. Suppose there exists $\{S_1,S_2,\cdots,S_r\}\subseteq \mcl{D}$ and $S_i\neq S,1\leq i\leq r$ such that for every $J\subseteq \{1,2,\cdots,r\}$ we have $\mid \us{j\in J}{\cup}S_j\mid \geq k+\mid J\mid$ and $\mid \us{i=1}{\os{r}{\cup}}S_j\cup S\mid < k+r+1$ then $S\in \mcl{D}$. This definition is motivated by the notion of independence in the Dilworth matroid $D(U_{k,n})$.
	
	We observe that the collection $\mcl{E}$ is concurrency closed.
	Now let $\mcl{D}\subs \mcl{E}$ be any arbitrary collection. Construct the concurrency closure $\overline{\mcl{D}}$ of $\mcl{D}$ as follows. 
	First set $\mcl{D}_0=\mcl{D}$ and add those elements $S\in \mcl{E}$ to $\mcl{D}_0$ if these $S$ satisfy the criterion mentioned above, to obtain $\mcl{D}_1$. Now construct $\mcl{D}_2$ from $\mcl{D}_1$ similarly and so on. 
	We have \equ{\mcl{D}_0=\mcl{D}\subsetneq \mcl{D}_1 \subsetneq \mcl{D}_2\subsetneq \cdots \subsetneq \mcl{D}_n=\overline{\mcl{D}}.}
	Since $\mcl{E}$ is a finite set we obtain $\overline{\mcl{D}}$ from $\mcl{D}_0$ in finitely many steps. Actually it can be shown that $\mcl{D}_1$ itself is concurrency closed and $\mcl{D}_1=\overline{\mcl{D}}$. This is because $\mcl{D}_1$ is set of all elements in $D(U_{k,n})$ which are dependent on $\mcl{D}_0$. Hence we get $\mcl{D}_1=\mcl{D}_2=\mcl{D}_3=\cdots=\ol{\mcl{D}}$.
\end{defn}
\begin{defn}[Base Collection]
	\label{defn:BC}
	~\\
	Let $n>k$ be two positive integers. Let \equ{\mcl{E}=\{\{i_1,i_2,\cdots,i_{k+1}\}\mid 1\leq i_1<i_2<\cdots<i_k<i_{k+1}\leq n\}} be the collection of all subsets of 
	cardinality $k+1$. Let $\mcl{D}\subs \mcl{E}$ be any arbitrary collection. 
	We say $\ti{\mcl{D}}$ is a base collection for $\mcl{D}$ if $\overline{\ti{\mcl{D}}}=\overline{\mcl{D}}$ and $\ti{\mcl{D}}$ is minimal, that is, if 
	$\mcl{D}'$ is any other collection such that $\overline{\mcl{D}'}=\overline{\mcl{D}}$ and $\mcl{D}'\subseteq \ti{\mcl{D}}$ then we have $\ti{\mcl{D}}=\mcl{D}'$.  We can actually show that all minimal bases have equal cardinality. This follows from the property of bases (a standard fact) in matroid theory.
\end{defn}
\begin{defn}[Construction of a Base for a Concurrency Closed Collection]
	\label{defn:BaseConstruction}
	Let $n>k$ be two positive integers. Let \equ{\mcl{E}=\{\{i_1,i_2,\cdots,i_{k+1}\}\mid 1\leq i_1<i_2<\cdots<i_k<i_{k+1}\leq n\}} be the collection of all subsets of 
	cardinality $k+1$. Let $\mcl{D}\subs \mcl{E}$ be a concurrency closed subcollection. We say there is a concurrency of order $m\geq k+1$ in $\mcl{D}$, if there exists a concurrency set $D\subs \{1,2,\cdots,n\}$
	of size $m$ such that all $\binom{m}{k+1}$ subsets of $D$ of size $k+1$ are in the collection $\mcl{D}$. Moreover $D$ should be maximal with respect to this property, that is, there does not exist a set $E \supsetneq D$
	of size more than $m$ such that all $\binom{\mid E \mid}{k+1}$ subsets of size $k+1$ are in the collection $\mcl{D}$.
	Let $m_1,m_2,\cdots,m_r$ be the orders of concurrencies that exist in $\mcl{D}$ with $m_i\geq k+1, 1\leq i\leq r$. Then the cardinality of a base collection $\mcl{D}'$ for $\mcl{D}$ is given by 
	\equ{(m_1-k)+(m_2-k)+\cdots+(m_r-k)=\bigg(\us{i=1}{\os{r}{\sum}}m_i\bigg)-rk.}
	Also see Corollary $3.6$ in C.~A.~Athanasiadis~\cite{MR1720104}.  Let $D_i\subs \{1,2,\cdots,n\}$ be the concurrency set of size $m_i$ which gives rise to the order $m_i$ concurrency in the concurrency closed subcollection $\mcl{D}$. Let $\mcl{S}=\{D_1,D_2,\cdots,D_r\}$. We have that 
	\equa{\#\big(\mcl{D}'&=Base\ of\ (\mcl{D})\big) = n-\dim\bigg(\us{\{i_1,i_2,\cdots,i_k,i_{k+1}\} \in \mcl{D}}{\bigcap} M_{\{i_1<i_2<\cdots<i_k<i_{k+1}\}}\bigg)\\&=\us{D\in \mcl{S}}{\sum}\gn(D)} 
	in the notation of Corollary 3.6 in~\cite{MR1720104} where $\gn(D)=\max(0,\mid D\mid-\ m)$ for $D\subs\{1,2,\cdots,n\}$. We also can show in this case that $\mcl{S}\in P(n,k)$ (See Theorem~\ref{theorem:Concurrency}). Here we do something more. We actually construct a base collection $\mcl{D}'$ for $\mcl{D}$.
	Let the concurrency sets be given by 
	\equ{D_i=\{j^i_1<j^i_2<\cdots<j^i_{m_i}\}, 1\leq i\leq r.}
	Then a base collection $\mcl{D}'$ for $\mcl{D}$ is given by 
	\equ{\{\{j^i_1,j^i_2,\cdots,j^i_k,j^i_l\}\mid k+1 \leq l \leq m_i,1\leq i\leq r\}.}
	This collection $\mcl{D}'$ has the required cardinality. We denote \equ{rank(\mcl{D})=\us{D\in \mcl{S}}{\sum}\gn(D) \ \ (\text{is the cardinality of any base collection of } \mcl{D}).}
\end{defn}
\begin{theorem}
	\label{theorem:Concurrency}
	Let $n>k$ be two positive integers. Let \equ{\mcl{E}=\{\{i_1,i_2,\cdots,i_{k+1}\}\mid 1\leq i_1<i_2<\cdots<i_k<i_{k+1}\leq n\}} be the collection of all subsets of 
	cardinality $k+1$. Let $\mcl{D}\subs \mcl{E}$ be a concurrency closed subcollection. Let $\mcl{S}=\{D_1,D_2,\cdots,D_r\}$ be the collection of sets of concurrencies in $\mcl{D}$. Then 
	\begin{enumerate}
		\item \begin{enumerate}[label=(\alph*)]
			\item $\mcl{S}\in P(n,k)$.
			\item $V_{\mcl{S}}=\us{E\subs D_i, \text{ for some }1\leq i\leq r,\mid E\mid=k+1}{\sum}V_E=\us{E\in \mcl{D}}{\sum}V_E=V_{\mcl{D}}$. 
		\end{enumerate}
		Moreover $\mcl{S}$ is uniquely determined with these two properties $(a),(b)$.
		\item Let the concurrency sets be given by $D_i=\{j^i_1<j^i_2<\cdots<j^i_{m_i}\}, 1\leq i\leq r$. Then a base collection $\mcl{D}'$ for $\mcl{D}$ is given by $\{\{j^i_1,j^i_2,\cdots,j^i_k,j^i_l\}\mid k+1 \leq l \leq m_i,1\leq i\leq r\}$. The
		\item The cardinality of any base of $\mcl{D}$ is $\us{D\in \mcl{S}}{\sum}\gn(D)$. 
	\end{enumerate}
\end{theorem}
\begin{proof}
	Let $\mcl{T}_0=\{T_1,T_2,\cdots,T_p\}\subseteq \mcl{D}$ be a base for the antichain $\mcl{D}$. So $V_{\mcl{T}_0}=\us{i=1}{\os{p}{\bigoplus}}V_{T_i}=V_{\mcl{D}}$. We have $\Gd(\mcl{F})\geq 0$ for all $\mcl{F}\subseteq \mcl{T}_0$ and $\mcl{T}_0$ is an antichain. So we apply Lemma~\ref{lemma:Antichain} repeatedly to obtain a collection $\mcl{T}\in P(n,k)$ such that $\mcl{T}_0\leq \mcl{T}$ and $V_{\mcl{T}}=V_{\mcl{T}_0}=V_{\mcl{D}}$. We prove that $\mcl{T}=\mcl{S}$. 
	
	Let $T\in \mcl{T}$ and $T=T_1\cup T_2\cup \cdots \cup T_l$ after renumbering the subsets $\mcl{T}_0$ and this union is obtained using Lemma~\ref{lemma:Antichain}. Then we have $\Gd(\{T_1\cup T_2\cup \cdots \cup T_l\})=0$. Let $E\subseteq T$ of cardinality $k+1$. Then we have $k+l+1>k+l= \mid T_1\cup T_2\cup \cdots \cup T_l\mid=\mid T\mid=\mid T_1\cup T_2\cup \cdots \cup T_l\cup E\mid$. Hence $\{E,T_1,\cdots,T_l\}$ is dependent and $\{T_1,\cdots,T_l\}$ is independent, that is, $V_{E}\subseteq \us{i=1}{\os{l}{\bigoplus}}V_{T_i}$. So $E\in \ol{\mcl{T}_0}=\mcl{D}$. This shows that $\mcl{P}_{k+1}(T)\subseteq \mcl{D}$ where $\mcl{P}_{k+1}(T)$ is the collection of all $(k+1)$-subsets of $T$. By the definition of concurrency orders and by the definition of $\mcl{S}$ there exists $1\leq i\leq r$ such that $T\subseteq D_i$.  So we have proved that $\mcl{T}\leq \mcl{S}$.

	Now we prove that $\mcl{S}\leq \mcl{T}$. Let $E\subseteq D \in \mcl{S}$ be a $(k+1)$-subset. Then there exists $\mcl{T}_0'=\{E=T_1',T_2',\cdots,T_p'\}$ a base of $\mcl{D}$. So $V_{\mcl{T}'_0}=\us{i=1}{\os{p}{\bigoplus}}V_{T'_i}=V_{\mcl{D}}$. We have $\Gd(\mcl{F})\geq 0$ for all $\mcl{F}\subseteq \mcl{T}'_0$ and $\mcl{T}'_0$ is an antichain. So we apply Lemma~\ref{lemma:Antichain} repeatedly to obtain a collection $\mcl{T}'\in P(n,k)$ such that $\mcl{T}'_0\leq \mcl{T}'$ and $V_{\mcl{T}'}=V_{\mcl{T}'_0}=V_{\mcl{D}}$.
	This implies $\gf(\mcl{T}')=V_{\mcl{D}}=\gf(\mcl{T})$ and $\mcl{T},\mcl{T}'\in P(n,k)$. So by injectivity of $\gf$ we have that $\mcl{T}=\mcl{T}'$. Now clearly $\{E\}=\{T_1'\}\leq \mcl{T}_0'\leq \mcl{T}'=\mcl{T}$ for every $(k+1)$-subset $E\subseteq D\in \mcl{S}$ . Using Lemma~\ref{lemma:PosetInequality} we conclude that $\{D\}\leq \mcl{T}$ as $T\in P(n,k)$ for every $D\in \mcl{S}$. This implies $\mcl{S}\leq \mcl{T}$. 
	So we have $\mcl{S}=\mcl{T}$ and have proved Theorem~\ref{theorem:Concurrency}(1).
	
	Now we prove $(2)$. Since $\{D_1,D_2,\cdots,D_r\}=\mcl{S}=\mcl{T}\in P(n,k)$ obtained by applying Lemma~\ref{lemma:Antichain} repeatedly to a base of $\mcl{D}$, the sum $\us{i=1}{\os{r}{\sum}}V_{D_i}=\us{i=1}{\os{r}{\oplus}}V_{D_i}$ is direct and if $D_i=\{j^i_1<j^i_2<\cdots<j^i_{m_i}\}, 1\leq i\leq r$ then the set $\mcl{D}'=\{\{j^i_1,j^i_2,\cdots,j^i_k,j^i_l\}\mid k+1 \leq l \leq m_i,1\leq i\leq r\}$ gives a base for $\mcl{D}$. This proves $(2)$.
	
	To prove $(3)$ we observe that the cardinality of $\mcl{D}'$ is $\us{D\in \mcl{S}}{\sum}\gn(D)$. Hence the cardinality of any base of $\mcl{D}$ is also given by the same value.   
	
	Hence Theorem~\ref{theorem:Concurrency} follows.
\end{proof}

Now we consider a new poset which is isomorphic to $P(n,k)$. 
\begin{defn}
	Let $C(n,k)$ be the collection of all concurrency closed subcollections of $\mcl{E}=\{\{i_1,i_2,\cdots,i_k,i_{k+1}\}\mid 1\leq i_1<\cdots<i_{k+1}\leq n\}$. Let $\mcl{D}_1,\mcl{D}_2$ be two concurrency closed subcollections of $\mcl{E}$. We say $\mcl{D}_1\leq \mcl{D}_2$ if $\mcl{D}_1\subseteq \mcl{D}_2$.
\end{defn}
Define a map $\psi:P(n,k)\lra C(n,k)$ given by \equ{\psi(\mcl{S})=\{E\subs\{1,2,\cdots,n\}\text{ such that } \mid E\mid =k+1,E\subs S \text{ for some }S\in \mcl{S}\}.}
In fact we will show that $\psi(\mcl{S})$ is concurrency closed and hence is an element of $C(n,k)$.
Define another map $\psi$ is given by $\gs:C(n,k)\lra P(n,k)$ defined as 
\equ{\gs(\mcl{D})=\mcl{S} \text{ where }\mcl{S} \text{ is the collection of sets of concurrencies in }\mcl{D}}
\begin{theorem}
	\label{theorem:PosetConcurrency}
	The maps $\psi$ and $\gs$ are poset isomorphisms and inverses of each other.
\end{theorem}
\begin{proof}
	First we prove $\psi(\mcl{S})$ is concurrency closed.
	Let $E$ be a $(k+1)$-subset which is dependent on $\psi(\mcl{S})$. So we have $V_E\subseteq V_{\mcl{S}}$. Then there exists a basis for $V_{\mcl{S}}$ whose corresponding $(k+1)$-subsets has the form $\mcl{T}_0=\{T_1=E,T_2,\cdots,T_p\}$. Now the hypothesis of Lemma~\ref{lemma:Antichain} is satisfied for  $\mcl{T}_0$. So by applying Lemma~\ref{lemma:Antichain} repeatedly there exists $\mcl{T}\in P(n,k)$ such that $\mcl{T}_0\leq \mcl{T}$ and $\gf(\mcl{T})=V_{\mcl{T}}=V_{\mcl{T}_0}=V_{\mcl{S}}=\gf(\mcl{S})$. So by injectivity of $\gf$ we get that $\mcl{T}=\mcl{S}$. Since $\{E\}\leq \mcl{T}_0\leq \mcl{T}=\mcl{S}$ there exists $S\in \mcl{S}$ such that $E\subseteq S$. So $E\in \psi(\mcl{S})$. So $\psi(\mcl{S})$ is concurrency closed.
	
	Now let $\psi(\mcl{S})=\mcl{D}$ and $\gs(D)=\mcl{T}$ then we get $V_{\mcl{S}}=V_{\mcl{D}}$ and $V_{\mcl{T}}=V_{\mcl{D}}$ by Theorem~\ref{theorem:Concurrency}(1)(b). Hence $\gf(\mcl{S})=\gf(\mcl{T})\Ra \mcl{S}=\mcl{T}$. This proves that $\gs\circ\psi=\mbb{1}_{P(n,k)}$.
	
	Let $\mcl{D}$ be concurrency closed and let $\gs(D)=\mcl{S}=\{D_1,D_2,\cdots,D_r\}$. Then every $(k+1)$-set $E\in \mcl{D}$ belongs to $\mcl{D}_i$ for some $1\leq i\leq r$ by the definition of sets of concurrencies in $\mcl{D}$. Hence $E\in \psi(\mcl{S})$. So $\mcl{D}\subseteq \psi(S)$. Similarly it clear that $\psi(\mcl{S})\subseteq \mcl{D}$. Hence we get $\psi\circ\gs=\mbb{1}_{C(n,k)}$. 
	
	Now we show that $\psi$ is order preserving. If $\mcl{S}\leq \mcl{T}$ then it is clear that $\psi(\mcl{S})\subseteq \psi(\mcl{T})$.
	Similarly if $\mcl{D}_1\subseteq \mcl{D}_2$ then by definition of sets of concurrencies in $\mcl{D}_1$ and $\mcl{D}_2$ we have $\gs(\mcl{D}_1)\leq \gs(\mcl{D}_2)$. So $\gs$ is also order preserving.
	
	This proves Theorem~\ref{theorem:PosetConcurrency}.
\end{proof}

\end{document}